\documentclass[12pt]{article}

\usepackage{amsfonts,mathrsfs}
\usepackage{amssymb,amsmath,amsbsy,amsthm}
\usepackage{dsfont}
\usepackage{blkarray}
\usepackage{tikz}
\usepackage{tkz-euclide}
\usepackage{tikz-cd}
\usetikzlibrary{patterns}
\usepackage{ae}
\usepackage{bm}
\usepackage{graphicx}
\usepackage{float}
\usepackage{cite}
 \usepackage{indentfirst}
\usepackage{lineno}
\usepackage{titlesec}
\usepackage{enumitem}
\usepackage{color}
\usepackage{aliascnt}
\usepackage{varioref}
\usepackage{hyperref}
\hypersetup{colorlinks=true,linkcolor=cyan,anchorcolor=cyan,citecolor=red,CJKbookmarks=True,
,urlcolor=cyan}
\usepackage{cleveref}

\numberwithin{equation}{section}
\def\rank{\mbox{\rm rank}}

\def\Nul{\mbox{\rm Nul}}

\def\sign{\mbox{\rm Sign}}
\def\Sign{\mbox{\rm sign}}
\def\face{\mbox{\rm Face}}

\def\pos{\mbox{\rm pos}}
\def\supp{\mbox{\rm supp}}
\def\relint{\mbox{\rm relint}}
\def\int{\mbox{\rm int}}

\def\aff{\mbox{\rm aff}}

\def\min{\mbox{\rm min}}

\def\And{\mbox{\rm ~and~}}
\def\Or{\mbox{\rm ~or~}}

\def\For{\mbox{\rm ~for~}}

\def\i{\mbox{\rm (\hspace{0.2mm}i\hspace{0.2mm})}\,}
\def\ii{\mbox{\rm (\hspace{-0.1mm}i\hspace{-0.2mm}i\hspace{-0.1mm})}\,}
\def\iii{\mbox{\rm (\hspace{-0.5mm}i\hspace{-0.3mm}i\hspace{-0.3mm}i\hspace{-0.5mm})}\,}
\def\iv{\mbox{\rm (\hspace{-0.5mm}i\hspace{-0.3mm}v\hspace{-0.5mm})}\,}

\def\Span{\mbox{\rm span}}
\def\Col{\mbox{\rm Col\,}}

\def\Sqcup{\textstyle\bigsqcup\limits}

\def\Sqcup{\textstyle\bigsqcup\limits}

\def\red{\color{red}}

\def\({\mbox{\rm (}}\def\){\mbox{\rm )}}

\makeatletter

\newcommand{\Rmnum}[1]{\expandafter\@slowromancap\romannumeral #1@}
\makeatother
\newtheorem{theorem}{Theorem}[section]
\newaliascnt{lemma}{theorem}
\newtheorem{lemma}[lemma]{Lemma}
\aliascntresetthe{lemma}

\newaliascnt{proposition}{theorem}
\newtheorem{proposition}[proposition]{Proposition}
\aliascntresetthe{proposition}

\newaliascnt{fact}{theorem}
\newtheorem{fact}[fact]{Fact}
\aliascntresetthe{fact}

\newaliascnt{definition}{theorem}
\newtheorem{definition}[definition]{Definition}
\aliascntresetthe{definition}

\newaliascnt{conjecture}{theorem}

\aliascntresetthe{conjecture}

\newaliascnt{corollary}{theorem}
\newtheorem{corollary}[corollary]{Corollary}
\aliascntresetthe{corollary}

\newaliascnt{claim}{theorem}

\aliascntresetthe{claim}

\newaliascnt{problem}{theorem}

\aliascntresetthe{problem}

\newaliascnt{question}{theorem}

\aliascntresetthe{question}

\newaliascnt{remark}{theorem}

\aliascntresetthe{remark}

\newaliascnt{example}{theorem}
\newtheorem{example}[example]{Example}
\aliascntresetthe{example}

\newaliascnt{notation}{theorem}

\aliascntresetthe{notation}

\begin{document}
\begin{center}
{\Large\bf
Deformed Intersections of Half-spaces
}\\ [7pt]
\end{center}
\vskip 3mm

\begin{center}
Houshan Fu$^{1}$, Boxuan Li$^{2}$, Chunming Tang$^{3}$ and Suijie Wang$^{4,*}$\\[8pt]

 $^{1,3}$School of Mathematics and Information Science\\
 Guangzhou University\\
Guangzhou 510006, Guangdong, P. R. China\\[12pt]

 $^{2,4}$School of Mathematics\\
 Hunan University\\
 Changsha 410082, Hunan, P. R. China\\[15pt]

$^{*}$Correspondence to be sent to: wangsuijie@hnu.edu.cn\\
Emails: $^{1}$fuhoushan@gzhu.edu.cn, $^{2}$liboxuan@hnu.edu.cn,  $^{3}$ctang@gzhu.edu.cn\\[15pt]
\end{center}

\vskip 3mm
\begin{abstract}
This paper is devoted to the classification problems concerning extended deformations of convex polyhedra and real hyperplane arrangements in the following senses: combinatorial equivalence of face posets, normal equivalence on normal fans of convex polyhedra, and sign equivalence on half-spaces.  The extended deformations of convex polyhedra arise from parallel translations of given half-spaces and hyperplanes, whose normal vectors give rise to the so-called ``derived arrangement'' proposed by Rota as well as Crapo in different forms. We show that two extended deformations of convex polyhedra are normally (combinatorially, as a consequence) equivalent if they are parameterized by the same open face of the derived arrangement. Note that these extended deformations are based on parallel translations of the given hyperplanes. It allows us to study three deformations of real hyperplane arrangements: parallel translations, conings, and elementary lifts, whose configuration spaces are parameterized by open faces of the derived arrangement. Consequently, it gives a characterization of the normal, combinatorial, and sign equivalences of those three deformations via the derived arrangement. Additionally, the relationships among these three equivalence relations are discussed, and several new descriptions of real derived arrangements associated with faces and sign vectors of real hyperplane arrangements are provided.
\vspace{1ex}\\
\noindent{\bf Keywords:} Convex polyhedron, hyperplane arrangement, deformation, parallel translation, coning, elementary lift, derived arrangement.\vspace{1ex}\\
{\bf MSC classes:} 52C35, 52B05, 52C40.
\end{abstract}
\section{Introduction}\label{Sec-1}
Convex polyhedra are very interesting geometric objects with rich combinatorial structures. Their study is closely linked to various combinatorial objects, such as hyperplane arrangements and (affine) oriented matroids. Recently, research on deformations of  polytopes is greatly flourishing, with new and exciting ideas continuously emerging. A deformation of a polytope $P$ is also a polytope obtained from $P$ by translating the facets of $P$ without passing a vertex.  As the most prominent example, the generalized permutohedra are deformations of a regular permutohedron, initially introduced by Edmonds in 1970 under the name of polymatroids \cite{Edmonds1970} and rediscovered by Postnikov in 2009 \cite{Postnikov2009}. Topics related to deformations of polytopes have attracted a lot of research attentions, such as Aguiar-Ardila \cite{Aguiar-Ardila2023}, Ardila-Sanchez \cite{Ardila-Sanchez2023}, Castillo-Liu \cite{Castillo-Liu2022}, Lee-Liu \cite{Lee-Liu2024}, Padrol-Pilaud-Poullot \cite{Padrol2022,Padrol2023,Padrol2023-1}, etc. Their primary focus is on studying the descriptions, combinatorial structures and algebraic structures of deformations and deformation cones for various interesting polytopes including permutohedra, nestohedra, graphical zonotopes, among others. However, our primary interest lies in investigating the classification problem concerning  normal and combinatorial equivalences on extended deformations of convex polyhedra that arise from parallel translations of given half-spaces and hyperplanes. Notably, these extended deformations correspond to the faces of parallel translations of a given linear arrangement. This correspondence enables us to explore a similar topic related to normal, combinatorial and sign equivalence classes in three types of deformations of  real hyperplane arrangements: parallel translation, coning and elementary lift. Since sign vectors of real hyperplane arrangements naturally induce (affine) oriented matroids in \cite{Bjorner1999}, these associated questions are also studied in (affine) oriented matroids.

The deformations of a polytope precisely form a polyhedral cone under dilation and Minkowski sum, called the deformation cone \cite{Postnikov2009}. Postnikov, Reiner and Williams \cite{Postnikov2008} offered five equivalent ways to characterize deformations of a simple polytope, which naturally lead to different descriptions for the deformation cone including vertex deformation cone, edge length deformation cone, facet deformation cone, among others. Most recently,  Castillo and Liu \cite{Castillo-Liu2022} not only extended their methods to general polytopes, but also introduced two additional equivalent definitions for the deformations of polytopes. Their definition  in \cite[Definition 2.2]{Castillo-Liu2022}, based on the idea of   ``moving facets without passing vertices", provided a crucial technique for parameterizing the deformation cone of a polytope using the height deformation cone.

Inspired by Castillo-Liu's work, we define and examine extended deformations of convex polyhedra. To this end, we always work with $m$ fixed non-zero vectors $\bm u_1,\bm u_2,\ldots,\bm u_m$ in the Euclidean space $\mathbb{R}^n$.  Let $U$ be a fixed $m\times n$ real matrix with row vectors $\bm u_1,\bm u_2,\ldots,\bm u_m$ in turn. For any subset $I\subseteq[m]$, denote by $U_I$ the submatrix of $U$ consisting of rows indexed by $I$, and by $\bm a_I$ the subvector of $\bm a\in\mathbb{R}^m$ according to the same convention. We define $(\bm a, I,J,K)$ as a tetrad, where $I,J,K$ is always a partition of the set $[m]$ and $\bm a\in\mathbb{R}^m$. Associated with each such tetrad $(\bm a, I,J,K)$, we consider the (possibly empty) convex polyhedron $P(\bm a,I,J,K)$ in $\mathbb{R}^n$ as follows:
\begin{equation*}\label{Polyhedral-Set}
P(\bm a,I,J,K):=\big\{\bm x\in\mathbb{R}^n\mid U_I\bm x= \bm a_I,\;U_J\bm x\le \bm a_J,\;U_K\bm x\ge \bm a_K\big\}.
\end{equation*}
\begin{definition}\label{Deformation-Polyhedra}
{\rm
Suppose $P(\bm a,I,J,K)$ is nonempty. A nonempty convex polyhedron $P$ is an {\em extended deformation} of $P(\bm a,I,J,K)$ if $P$ can be written as $P=P(\bm b,I,J,K)$ for some $\bm b\in\mathbb{R}^m$. In this sense, the vector $\bm b$ is called an {\em extended deforming vector} for $P$.
}
\end{definition}

Clearly each extended deformation of a convex polyhedron is parameterized by the corresponding extended deforming vector, which is not necessarily unique as illustrated in \autoref{Example1}. However, our definition provides a optimal and particularly precise framework for characterizing the normal, combinatorial and sign equivalence classes in the three deformations of real hyperplane arrangements as discussed in \autoref{Sec-4}.

In addition, the deformations of a polytope can equivalently be described by normal fans. Namely, a polytope $Q$ is a deformation of a polytope $P$ if the normal fan $\mathcal{N}(Q)$ is a coarsening of the normal fan $\mathcal{N}(P)$ in \cite[Proposition 2.6]{Castillo-Liu2022}. To study the deformation cone of a Coxeter permutohedron, the authors in \cite{Ardila2020} extended the concept of deformation from polytopes to convex polyhedra, and also introduced  `extended deformation' of  convex polyhedra. Specifically, a convex polyhedron $Q$ is an `extended deformation' of a convex polyhedron $P$ if  the normal fan $\mathcal{N}(Q)$ coarsens a convex subfan of $\mathcal{N}(P)$ in \cite[Definition 2.3]{Ardila2020}. As stated in \autoref{Example1}, the two mentioned definitions of the extended deformation of convex polyhedra are not completely equivalent. According to our definition, if $Q$ is an extended deformation of a convex polyhedron $P$, then the normal fan $\mathcal{N}(Q)$  either coarsens a convex subfan of the normal fan $\mathcal{N}(P)$ or a convex subfan of $\mathcal{N}(Q)$ is refined by $\mathcal{N}(P)$.

Two convex polyhedra $P$ and $Q$ are said to be {\em normally equivalent} ({\em combinatorially equivalent}, resp.)  if they have the same normal fan (their face posets are isomorphic, resp.). Another terms that are sometimes used instead of `normally equivalent' are `analogous' \cite{Alexandrov1937} that is a very important technique for proving the Alexandrov–Fenchel inequality or `strongly isomorphic' and `strongly combinatorially equivalent' \cite{McMullen1973, McMullen1996}. From this perspective, the interior of the deformation cone of a polytope $P$ precisely corresponds to a special class of  normally equivalent polytopes consisting of those deformations with the same normal fan as $P$, referred to as the type cone \cite{McMullen1973}. In general, \cite[Theorem 2.4]{Poullot2024} implied that deforming vectors in each open face of the height deformation cone exactly determine a class of  normally equivalent polytopes, where these polytopes share the same normal fan. This phenomenon triggers the following result, showing that the convex polyhedra  parameterized by the same open face of the derived arrangement $\delta\mathcal{A}_{\bm o}$ are normally and combinatorially equivalent. The derived arrangement  $\delta\mathcal{A}_{\bm o}$ is defined in \eqref{derived-arrangement}.
\begin{theorem}\label{Normally-Combinatorially}
Let $F$ be a face of $\delta\mathcal{A}_{\bm o}$ and $\bm a,\bm b\in\relint(F)$. Then for any partition $I,J,K$ of $[m]$,
$P(\bm a,I,J,K)$ and $P(\bm b,I,J,K)$ are normally and combinatorially equivalent.
\end{theorem}

Note that a face in a real hyperplane arrangement is a convex polyhedron. So the term `normally equivalent' is naturally extended to real hyperplane arrangements.
Roughly speaking, two hyperplane arrangements $\mathcal{A}$ and $\mathcal{A}'$ in $\mathbb{R}^n$ are {\em normally equivalent} if there is an order-preserving bijection between their faces such that each corresponding face pair shares the same normal fan. We say that two real hyperplane arrangements $\mathcal{A}$ and $\mathcal{A}'$ are {\em sign equivalent} ({\em combinatorially equivalent}, {\em semi-lattice equivalent}, resp.) if their faces have the same signs (their face posets are isomorphic, their semi-lattices are isomorphic, resp.). Detailed definitions can be found in \autoref{Equivalence-Arrangement}. The result in \eqref{Four-Relations} quickly tells us that the normal and sign equivalences are stronger than the combinatorial and semi-lattice equivalences for real hyperplane arrangements.

Our other main goal is to address classification problems in the three deformations of real hyperplane arrangements: parallel translation, coning and  elementary lift. We always regard $\mathcal{A}_{\bm o}$ as the linear arrangement
\[
\mathcal{A}_{\bm o}:=\big\{H_{\bm u_i}:\langle\bm u_i,\bm x\rangle=0\mid i=1,2,\ldots,m\big\}
\]
in $\mathbb{R}^n$ that may be a multi-arrangement. For any $\bm a=(a_1,\ldots,a_m)\in\mathbb{R}^m$, we say that
\begin{itemize}
\item a {\em parallel translation} $\mathcal{A}_{\bm a}$ of $\mathcal{A}_{\bm o}$ is a hyperplane arrangement  in $\mathbb{R}^n$ defined as
\[
\mathcal{A}_{\bm a}:=\big\{H_{\bm u_i,a_i}:\langle\bm u_i,\bm x\rangle=a_i\mid  i=1,2,\ldots,m\big\};
\]
\item a {\em coning} $c\mathcal{A}_{\bm a}$ of $\mathcal{A}_{\bm a}$ is a linear arrangement in $\mathbb{R}^{n+1}$  given by
\[
c\mathcal{A}_{\bm a}:=\big\{ c H_{\bm u_i,a_i}:\bm \langle\bm u_i,\bm x\rangle+a_i\cdot x_{n+1}=0\mid  i=1,2,\ldots,m\big\}\Sqcup\{K_0: x_{n+1}=0\};
\]
\item an {\em elementary lift} $\mathcal{A}^{\bm a}$ of $\mathcal{A}_{\bm o}$ is a linear arrangement in  $\mathbb{R}^{n+1}$ defined to be
\[
\mathcal{A}^{\bm a}:=\big\{H_{(\bm u_i,a_i)}:\langle\bm u_i,\bm x\rangle+a_i\cdot x_{n+1}=0\mid  i=1,2,\ldots,m\big\}.
\]
\end{itemize}

Recently, Chen, Fu and Wang in \cite{Fu-Wang2024,Chen-Fu-Wang2021} provided a full description of the semi-lattice equivalence classes in the three deformations by the derived arrangement $\delta\mathcal{A}_{\bm o}$. Here we uniformly characterize the normal or combinatorial classes in the corresponding deformations using the same geometric structure $\delta\mathcal{A}_{\bm o}$.

\begin{theorem}\label{Three-Type-Normal-Combinatorial}
Let $F$ be a face of $\delta\mathcal{A}_{\bm o}$ and $\bm a,\bm b\in\relint(F)$. Then
\begin{itemize}
\item [\i] the parallel translations $\mathcal{A}_{\bm a}$ and $\mathcal{A}_{\bm b}$ are normally and combinatorially equivalent;
\item [\ii] the conings $c\mathcal{A}_{\bm a}$ and $c\mathcal{A}_{\bm b}$ are combinatorially equivalent;
\item [\iii] the elementary lifts $\mathcal{A}^{\bm a}$ and $\mathcal{A}^{\bm b}$ are combinatorially equivalent.
\end{itemize}
\end{theorem}

Furthermore, we also show that every open face of the derived arrangement $\delta\mathcal{A}_{\bm o}$ corresponds precisely to a unique sign equivalence class of the parallel translations $\mathcal{A}_{\bm a}$, conings $c\mathcal{A}_{\bm a}$ and elementary lifts $\mathcal{A}^{\bm a}$.
\begin{theorem}
Let $\bm a,\bm b\in\mathbb{R}^m$.
Then the following are equivalent:
\begin{itemize}\label{Three-Type-Sign}
\item [\i]  $\bm a$ and $\bm b$ belong to the same open face of $\delta\mathcal{A}_{\bm o}$;
\item [\ii] the parallel translations $\mathcal{A}_{\bm a}$ and $\mathcal{A}_{\bm b}$ are sign equivalent;
\item [\iii] the conings $c\mathcal{A}_{\bm a}$ and $c\mathcal{A}_{\bm b}$ are sign equivalent;
\item [\iv] the elementary lifts $\mathcal{A}^{\bm a}$ and $\mathcal{A}^{\bm b}$ are sign equivalent.
\end{itemize}
\end{theorem}

Notice that the derived arrangement $\delta\mathcal{A}_{\bm o}$ decomposes the ambient space $\mathbb{R}^m$ into finitely many open faces, that is,
\begin{equation}\label{Whole-Space-Partition}
\mathbb{R}^m=\bigsqcup_{F\in \mathcal{F}(\delta\mathcal{A}_{\bm o})} \relint(F),
\end{equation}
where $\mathcal{F}(\delta\mathcal{A}_{\bm o})$ is the face poset of $\delta\mathcal{A}_{\bm o}$ and $\sqcup$ denotes disjoint union. Then \autoref{Three-Type-Normal-Combinatorial} and \autoref{Three-Type-Sign} show that each of the configuration spaces of  parallel translation $\mathcal{A}_{\bm a}$, coning $c\mathcal{A}_{\bm a}$ and elementary lift $\mathcal{A}^{\bm a}$ parameterized by $\bm a\in\mathbb{R}^m$, is partitioned into at most $|\mathcal{F}(\delta\mathcal{A}_{\bm o})|$ equivalence classes based on any one of the combinatorial, sign or normal equivalences.

Additionally, an (affine) oriented matroid, as described in \cite{Bjorner1999}, uses the covector axiomatization to model the combinatorial structure of sign vectors of real hyperplane arrangements. This allows us to  address analogous questions in (affine) oriented matroids in \autoref{Sec-5}.

The derived arrangement is a key tool in our study, originating from Crapo's introduction of the `geometry of circuits' to characterize the matroid $M(m,n,\mathcal{C})$ of circuits of the configuration $\mathcal{C}$ of $m$ generic points in $\mathbb{R}^n$. Subsequently, Manin-Schechtman \cite{Manin-Schechtman1989} defined the same object as higher braid arrangements, known as `discriminantal arrangement'. Most recently, the authors in \cite{Chen-Fu-Wang2021} introduced derived arrangements as a geometric model of derived matroids, and further explained that the derived arrangement can be viewed as a generalization of the discriminant arrangement.  The concept of derived matroid was first proposed by Rota \cite{Rota1971}, and explicitly defined for binary matroids by Longyear \cite{Longyear1980} and for general representable matroids by Oxley-Wang \cite{Oxley-Wang2019}.  In \autoref{Sec-6}, we also achieve several alternative characterizations for derived arrangements associated with faces and sign vectors of real arrangements.
 \paragraph{Organization.}
\autoref{Sec-2} sets necessary foundations on convex polyhedra and hyperplane arrangements. In \autoref{Sec-3}, we study extended deformations of convex polyhedra, and prove \autoref{Normally-Combinatorially}.  \autoref{Sec-4} is devoted to showing \autoref{Three-Type-Normal-Combinatorial} and \autoref{Three-Type-Sign}, and also establishes the relationships among the combinatorial, semi-lattice, normal and sign equivalence relations on real arrangements.  \autoref{Sec-5} further addresses similar questions in (affine) oriented matroids depending on \autoref{Sec-4}. Finally, we define the sign operator and face operator associated with the faces and sign vectors of real arrangements, and use them to provide several alternative characterizations of real derived arrangements in \autoref{Sec-6}.

\section{Preliminaries}\label{Sec-2}
In this section, we review necessary notations and definitions on convex polyhedra and real hyperplane arrangements.
\paragraph{2.1. Notations.} Here we shall fix our notations. For any non-zero vector $\bm u\in\mathbb{R}^n$ and real number $a\in\mathbb{R}$, notations $H_{\bm u}, H_{\bm u,a}, H_{\bm u,a}^-$, $H_{\bm u,a}^+$, $\mathring{H}_{\bm u,a}^-$ and $\mathring{H}_{\bm u,a}^+$  refer to a linear hyperplane
\[
H_{\bm u}:=\{\bm x\in\mathbb{R}^n:\langle\bm u,\bm x\rangle=0\},
\]
an affine hyperplane
\[
H_{\bm u,a}:=\{\bm x\in\mathbb{R}^n:\langle\bm u,\bm x\rangle=a\},
\]
the closed half-spaces
\[
H_{\bm u,a}^-:=\{\bm x\in\mathbb{R}^n:\langle\bm u,\bm x\rangle\le a\},\quad H_{\bm u,a}^+:=\{\bm x\in\mathbb{R}^n:\langle\bm u,\bm x\rangle\ge a\}\]
and the open half-spaces
\[
\mathring{H}_{\bm u,a}^-:=\{\bm x\in\mathbb{R}^n:\langle\bm u,\bm x\rangle<a\},\quad\mathring{H}_{\bm u,a}^+:=\{\bm x\in\mathbb{R}^n:\langle\bm u,\bm x\rangle>a\}
\]
in turn, where $\langle\cdot,\cdot\rangle$ denotes the standard scalar product in $\mathbb{R}^n$. For any $\bm x=(x_1,\ldots,x_m)\in\mathbb{R}^m$, we denote by $\supp(\bm x)$  the support of $\bm x$ given by
\[
\supp(\bm x):=\big\{i\in[m]:x_i\ne0\big\},
\]
where $[m]:=\{1,2,\ldots,m\}$. The relative interior of a subset $A$ of a topological space is denoted by $\relint(A)$. For convenience, we always assume $\bm a=(a_1,\ldots,a_m)$ and $\bm b=(b_1,\ldots,b_m)$ for any $\bm a,\bm b\in\mathbb{R}^m$.

\paragraph{2.2. Convex polyhedra.}
Terminology on convex polyhedra mainly refers to books \cite{Ewald1996,Schneider2014}. A {\em convex polyhedron} $P$ in $\mathbb{R}^n$ is an intersection of finitely many closed half-spaces. Particularly, if these closed half-spaces pass all through the origin $\bm o$, then the convex polyhedron $P$ is called a {\em polyhedral cone}. A bounded convex polyhedron is called a {\em polytope}. A hyperplane $H$ in $\mathbb{R}^n$ is called a {\em supporting hyperplane} of $P$ if $P\cap H\ne \emptyset$ and $P$ is contained in one of the closed half-spaces $H^-$ and $H^+$. In this case, we call the set $P\cap H$ a {\em face} and $\relint(P\cap H)$ an {\em open face} of $P$ respectively. By convention, $\emptyset$ and $P$ are {\em improper faces} of $P$.  Faces of convex polyhedra form an important poset. A {\em face poset} $\mathcal{F}(P)$ of $P$ is a set of all faces, partially ordered by inclusion. A member $F$ in $\mathcal{F}(P)$ is called a {\em $k$-face} if
\[
\dim(F):=\dim\big(\aff(P)\big)=k,
\]
where $\aff(P)$ denotes the minimum affine subspace containing $P$ in $\mathbb{R}^n$. The $(\dim(P)-1)$-faces of $P$ are called {\em facets}. By convention, $\dim(\emptyset)=-1$.  

A {\em support function} $h_P(\cdot):\mathbb{R}^n\to(-\infty,\infty]$ of $P$ is defined as
\[
h_P(\bm u):=\sup\big\{\langle\bm u,\bm x\rangle:\bm x\in P\big\},\quad \bm u\in\mathbb{R}^n.
\]
By convention, $H_{\bm 0}:=\mathbb{R}^n$. Let $F(P,\bm u):=H_{\bm u,h_P(\bm u)}\cap P$. If $h_P(\bm u)=\infty$, both sets $H_{\bm u,h_P(\bm u)}$ and $F(P,\bm u)$ are empty. Otherwise, the hyperplane $H_{\bm u,h_P(\bm u)}$ is a supporting hyperplane of $P$. In this sense, $F(P,\bm u)$ is exactly the corresponding face of $P$.

Let $P$ be a convex polyhedron in $\mathbb{R}^n$. The {\em normal cone} $N_P(\bm x)$ of $P$ at $\bm x$ is defined by
\[
N_P(\bm x):=\big\{\bm u\in\mathbb{R}^n\mid \bm x\in H_{\bm u,h_P(\bm u)}\big\}.
\]
When $\bm x$ and $\bm y$ are two relative interior points of a face $F\in\mathcal{F}(P)$,  clearly $N_P(\bm x)=N_P(\bm y)$. For this reason, the {\em normal cone}  $N_P(F)$ of $P$ at a nonempty face $F$ is defined as
\[
N_P(F):=N_P(\bm x)
\]
for some $\bm x\in\relint(F)$. The {\em normal fan} $\mathcal{N}(P)$ of $P$ is the collection of its normal cones, i.e.,
\[
\mathcal{N}(P):=\big\{N_P(F)\mid F\mbox{ is a nonempty face of } P\big\}.
\]
The {\em support} of $\mathcal{N}(P)$ is the union of normal cones in $\mathcal{N}(P)$, denoted by $|\mathcal{N}(P)|$.

We now introduce the normal and combinatorial equivalence relations on convex polyhedra, which will be discussed later.
\begin{definition}\label{Two-Equivalence-Polyhedra}
{\rm Two convex polyhedra $P$ and $Q$ in $\mathbb{R}^n$ are said to be:
\begin{itemize}
  \item {\em combinatorially equivalent} if their face posets are isomorphic, that is $\mathcal{F}(P)\cong\mathcal{F}(Q)$;
  \item {\em normally equivalent} if they have the same normal fan, that is $\mathcal{N}(P)=\mathcal{N}(Q)$.
\end{itemize}
}
\end{definition}
\paragraph{2.3. Real hyperplane arrangements.}
Terminology on hyperplane arrangements mainly refers to the literature \cite{Stanley}. A {\em real hyperplane arrangement} $\mathcal{A}$ is a finite collection of hyperplanes in an $n$-dimensional Euclidean space  $\mathbb{R}^n$. When the intersection of all hyperplanes in $\mathcal{A}$ is nonempty, $\mathcal{A}$ is called a {\em central arrangement} and the intersection is referred to as a {\em center}. In particular, $\mathcal{A}$ is {\em linear} if the center contains the origin $\bm o$.  For any $X\in L(\mathcal{A})$, the {\em localization} of $\mathcal{A}$ at $X$ is the subarrangement
\[
\mathcal{A}_X:=\{H\in\mathcal{A}\mid X\subseteq H\},
\]
and the {\em restriction} $\mathcal{A}/X$ of $\mathcal{A}$ on $X$ is a hyperplane arrangement in $X$ defined by
\[
\mathcal{A}/X:=\{H\cap X\ne\emptyset\mid  H\in \mathcal{A}\setminus\mathcal{A}_X\}.
\]

There are two important posets related to real hyperplane arrangements: intersection poset (semi-lattice) and face poset. The {\em intersection poset} or {\em  semi-lattice} $L(\mathcal{A})$ consists of all nonempty intersections of some hyperplanes in $\mathcal{A}$, ordered by the reverse inclusion, and including $\mathbb{R}^n:=\bigcap_{H\in\emptyset}H$ as the minimum element. In general, $L(\mathcal{A})$ is a semi-lattice, and it is a lattice if and only if $\mathcal{A}$ is central. The {\em complement} of $\mathcal{A}$ is defined as
\[
M(\mathcal{A}):=\mathbb{R}^n\setminus\bigcup_{H\in \mathcal{A}}H.
\]
It consists of finitely many connected components, called {\em regions} of $\mathcal{A}$. More generally, for any $k$-dimensional element $X$ in $L(\mathcal{A})$, each region of $\mathcal{A}/X$ is a {\em $k$-dimensional open face} of $\mathcal{A}$, and its closure is referred to as a {\em $k$-dimensional face} of $\mathcal{A}$. The {\em face poset} $\mathcal{F}(\mathcal{A})$ of $\mathcal{A}$ is the set of all faces, partially ordered by inclusion. It serves as a bridge between real hyperplane arrangements and (affine) oriented matroids, see \autoref{Sec-5} for details.

A hyperplane arrangement $\mathcal{A}=\{H_1,\ldots,H_m\}$ in $\mathbb{R}^n$ decomposes the ambient space $\mathbb{R}^n$ into a finite number of relatively open topological cells (open faces). Each open face is completely determined by its position information relative to every hyperplane $H_i$: it either lies on the hyperplane, lies on its positive side, or lies on its negative side. This leads us to associate a sign vector with every (open) face. For any point $\bm x \in\mathbb{R}^n$, a {\em sign vector} of $\bm x$ induced by $\mathcal{A}$ is defined as
\[
\Sign_{\mathcal{A}}(\bm x):=\big(\Sign_{\mathcal{A}}(\bm x)_1,\Sign_{\mathcal{A}}(\bm x)_2,\ldots,\Sign_{\mathcal{A}}(\bm x)_m\big),
\]
where the $i$-th component is determined as follows:
\begin{equation*}\label{Point-Sign}
\Sign_{\mathcal{A}}(\bm x)_i:=
\begin{cases}
   +, & \mbox{if }\quad\bm x\in \mathring{H}_i^+, \\
   0, & \mbox{if }\quad\bm x\in H_i,\\
   -, & \mbox{if }\quad\bm x\in \mathring{H}_i^-,
\end{cases}
\end{equation*}
for $1\le i\le m$. This sign vector encodes the position of $\bm x$ with respect to each hyperplane. Thus all points lying in an open face of $\mathcal{A}$ have the same sign vector. For this reason, we define a {\em sign vector} of a face $F\in\mathcal{F}(\mathcal{A})$ to be
\begin{equation}\label{Face-Point-Sign}
\Sign_{\mathcal{A}}(F):=\Sign_{\mathcal{A}}(\bm x)
\end{equation}
for some $\bm x\in\relint(F)$. Go one step further, we define a {\em sign} of $\mathcal{A}$ as
\[
\Sign(\mathcal{A}):=\big\{\Sign_{\mathcal{A}}(\bm x)\mid \bm x\in\mathbb{R}^n\big\}.
\]
Alternatively, $\Sign(\mathcal{A})$ can be represented as the following form
\begin{equation}\label{Arrangement-Sign}
\Sign(\mathcal{A})=\big\{\Sign_{\mathcal{A}}(F)\mid F\in\mathcal{F}(\mathcal{A})\big\}.
\end{equation}

We now define four types of equivalence relations on real hyperplane arrangements, which will be studied in \autoref{Sec-4}.
\begin{definition}\label{Equivalence-Arrangement}
{\rm
Let $\mathcal{A}=\{H_1,\ldots,H_m\}$ be a hyperplane arrangement in $\mathbb{R}^n$, and $\mathcal{A}'=\{H_1',\ldots,H_m'\}$ be a hyperplane arrangement in $\mathbb{R}^{n'}$. Then $\mathcal{A}$ and $\mathcal{A}'$ are said to be:
\begin{itemize}
\item {\em semi-lattice equivalent} if their semi-lattices are isomorphic, that is $L(\mathcal{A})\cong L(\mathcal{A}')$;
\item {\em combinatorially equivalent} if their face posets are isomorphic, that is $\mathcal{F}(\mathcal{A})\cong \mathcal{F}(\mathcal{A}')$;
\item {\em sign equivalent} if they have the same sign, that is $\Sign(\mathcal{A})=\Sign(\mathcal{A}')$.
\end{itemize}
When $n=n'$, $\mathcal{A}$ and $\mathcal{A}'$ are said to be:
\begin{itemize}
\item {\em normally equivalent} if there is an order-preserving bijection $\Psi$ from $\mathcal{F}(\mathcal{A})$ to $\mathcal{F}(\mathcal{A}')$ such that $F$ and $\Psi(F)$ are normally equivalent for any face $F$ of $\mathcal{A}$. In this sense, the map $\Psi$ is called a {\em normal map} from $\mathcal{A}$ to $\mathcal{A}'$.
\end{itemize}
}
\end{definition}

\paragraph{2.4. Derived arrangement.}
The derived arrangement is a key ingredient to obtain our results. Associated with the linear arrangement $\mathcal{A}_{\bm o}$, let
\begin{equation}\label{circuit}
\mathcal{C}(\mathcal{A}_{\bm o}):=\big\{C\subseteq[m]\mid\{\bm u_i\mid i\in C\}{\;\rm is\;a\;minimum\;linearly\;dependent\;set}\big\}.
\end{equation}
Each member of $\mathcal{C}(\mathcal{A}_{\bm o})$ is said to be a {\em circuit} of $[m]$ with respect to $\mathcal{A}_{\bm o}$. It is clear that for every $C\in\mathcal{C}(\mathcal{A}_{\bm o})$, there is a unique vector $\bm c^C=(c_1,c_2,\ldots,c_m)\in\mathbb{R}^m$ (up to a nonzero constant multiple) such that
\begin{equation}\label{circuit-vector}
\sum_{i=1}^{m}c_i\bm u_i=\bm0\quad \And\quad c_i\ne 0\Longleftrightarrow i\in C,
\end{equation}
called  a {\em circuit vector} of $C$. The {\em derived arrangement} $\delta\mathcal{A}_{\bm o}$ of $\mathcal{A}_{\bm o}$ is a linear arrangement in $\mathbb{R}^m$ defined as
\begin{equation}\label{derived-arrangement}
\delta\mathcal{A}_{\bm o}:=\big\{H_{\bm c^C}:\langle\bm c^C, \bm x\rangle=0\mid C\in \mathcal{C}(\mathcal{A}_{\bm o})\big\}.
\end{equation}
\section{Extended deformations of convex polyhedra}\label{Sec-3}
In this section, we show that all extended deformations of a convex polyhedron $P$ retain the same boundedness as $P$ in \autoref{Cone}, and also prove \autoref{Normally-Combinatorially}.
\subsection{Boundedness}\label{Sec-3-1}
We start by introducing Farkas' lemma which serves as a crucial tool for determining both the solvability of inequality systems and the boundedness of convex  polyhedra.
\begin{lemma}[\cite{Kuhn1956}, Farkas' lemma]\label{Farkas-lemma} Let $A$ be a $m\times n$ real matrix and $\bm a\in\mathbb{R}^m$. Then exactly one of the following statements is true
\begin{itemize}
  \item [{\rm1.}] there exists a vector $\bm x\in\mathbb{R}^n$ satisfying $A\bm x\le\bm a$,
  \item [{\rm2.}] there exists a row vector $\bm y\in\mathbb{R}_{\ge 0}^m$ satisfying $\bm yA=\bm 0$ and $\langle \bm y,\bm a\rangle<0$.
\end{itemize}
\end{lemma}
The following general Farkas' lemma for equality and (strict) inequality constraints can be derived from \autoref{Farkas-lemma}.
\begin{proposition}\label{General-Farkas-lemma}
Let $A_i$ be a $m_i\times n$  real matrix and $\bm a_i\in\mathbb{R}^{m_i}$ for $i=1,\ldots,5$.  Then exactly one of the following  statements is true
 \begin{itemize}
 \item[{\rm1.}] there exists a vector $\bm x\in\mathbb{R}^n$ satisfying
 \[
 A_1\bm x=\bm a_1,\quad A_2\bm x\le\bm a_2,\quad A_3\bm x<\bm a_3,\quad A_4\bm x\ge\bm a_4\quad\And\quad A_5\bm x>\bm a_5,
 \]
 \item[{\rm2.}] there exist row vectors $\bm y_i\in\mathbb{R}^{m_i}$ for $i=1,\ldots,5$ satisfying that
 \[
 \bm y_i\ge\bm 0\;(i=1,2),\quad\bm y_j\le\bm 0\;(j=4,5),\quad\sum_{i=1}^{5}\bm y_iA_i=\bm 0,
\]
and
\[
\sum_{i=1}^{5}\langle \bm y_i,\bm a_i\rangle<0\quad\Or\quad\sum_{i=1}^{5}\langle \bm y_i,\bm a_i\rangle=0\And \bm y_3\ne\bm 0 \Or \bm y_5\ne\bm 0.
\]
 \end{itemize}
\end{proposition}
\begin{proof}
Clearly the statement 1 is equivalent to that there exist $\bm x\in\mathbb{R}^n$ and $t\in\mathbb{R}$ satisfying
\[
t\ge1,\quad A_1\bm x= t\bm a_1,\quad A_2\bm x\le t\bm a_2,\quad A_3\bm x\le t\bm a_3-\bm 1,\quad -A_4\bm x\le -t\bm a_4\;\And\;-A_5\bm x\le -t\bm a_5-\bm 1.
\]
Namely, there exist $(\bm x,t)\in\mathbb{R}^{n+1}$ satisfying
\begin{equation}\label{General-Inequality-System}
 \begin{pmatrix}
 \bm 0&-1\\
 A_1&-\bm a_1\\
 -A_1&\bm a_1\\
 A_2&-\bm a_2\\
 A_3&-\bm a_3\\
  -A_4&\bm a_4\\
   -A_5&\bm a_5
 \end{pmatrix}
 \begin{pmatrix}
 \bm x\\
 t
 \end{pmatrix}\le
 \begin{pmatrix}
 -1\\
 \bm 0\\
 \bm 0\\
 \bm 0\\
 -\bm 1\\
 \bm 0\\
 -\bm 1
 \end{pmatrix}.
\end{equation}
For convenience, we denote the matrix on the left by $\tilde{A}$ and the vector on the right by $\tilde{\bm a}$. From \autoref{Farkas-lemma}, the system in \eqref{General-Inequality-System} is infeasible if and only if there exists a row vector
\[
\tilde{\bm y}=(y_0,\bm y_1^+,\bm y_1^-,\bm y_2,\bm y_3,\bm y_4^-,\bm y_5^-)\in\mathbb{R}^{1+m_1+m_1+m_2+m_3+m_4+m_5}
\]
satisfying
\begin{equation}\label{General-Inequality-System1}
\tilde{\bm y}\ge\bm 0,\quad\tilde{\bm y}\tilde{A}=\bm 0\quad\And\quad \langle\tilde{\bm y},\tilde{\bm a}\rangle<0.
\end{equation}
Let $\bm y_1=\bm y_1^+-\bm y_1^-$, $\bm y_4=-\bm y_4^-$ and $\bm y_5=-\bm y_5^-$. Then \eqref{General-Inequality-System1} becomes
\begin{equation*}
y_0\ge0,\quad\bm y_i\ge \bm 0\;(i=2,3),\quad \bm y_j\le\bm 0\;(j=4,5),\quad\sum_{i=1}^{5}\bm y_iA_i=\bm 0
\end{equation*}
and
\[
\sum_{i=1}^{5}\langle\bm y_i,\bm a_i\rangle=-y_0,\quad -\langle\bm y_3,\bm 1\rangle+\langle\bm y_5,\bm 1\rangle<y_0.
\]
If $y_0>0$, then $\sum_{i=1}^{5}\langle\bm y_i,\bm a_i\rangle<0$ and $-\langle\bm y_3,\bm 1\rangle+\langle\bm y_5,\bm 1\rangle\le 0<y_0$ always holds.
If $y_0=0$, then $\sum_{i=1}^{5}\langle\bm y_i,\bm a_i\rangle=0$ and $-\langle\bm y_3,\bm 1\rangle+\langle\bm y_5,\bm 1\rangle<0$ means $\bm y_3\ne\bm 0$ or $\bm y_5\ne\bm0$.  Then we obtain that the system in \eqref{General-Inequality-System} is infeasible if and only if the statement 2 holds. So we complete this proof.
\end{proof}

Recall from \eqref{Whole-Space-Partition} that the ambient space $\mathbb{R}^m$ are partitioned into the adjoint union of open faces of the derived arrangement $\delta\mathcal{A}_{\bm o}$. This demonstrates that for each member $\bm x\in\mathbb{R}^m$, there is a unique face $F$ of $\delta\mathcal{A}_{\bm o}$ satisfying ${\bm x}\in \relint(F)$, denoted by $F_{\bm x}$. Indeed, $F_{\bm x}$ is the inclusion-minimum face of $\mathcal{F}(\delta\mathcal{A}_{\bm o})$ containing ${\bm x}$. The following two lemmas are crucial for establishing the close connections among faces of extended deformations of convex polyhedra.

\begin{lemma}\label{Nonempty}
Let $P(\bm a,I,J,K)\ne\emptyset$ and $F_{\bm a}\in\mathcal{F}(\delta\mathcal{A}_{\bm o})$. Then $P(\bm b,I,J,K)\ne\emptyset$ for any $\bm b\in F_{\bm a}$.
\end{lemma}
\begin{proof}
We give a proof  by contradiction. Suppose $P(\bm b,I,J,K)=\emptyset$. This is equivalent to the insolvability of  the following inequality systems
\begin{equation}\label{Inequality-System}
U_I\bm x=\bm b_I,\quad U_J\bm x\le\bm b_J\;\And\;U_K\bm x\ge\bm b_K.
\end{equation}
Applying \autoref{General-Farkas-lemma} to \eqref{Inequality-System}, it is further equivalent to that there exists a row vector $\bm y=(y_1,\ldots,y_m)\in\mathbb{R}^m$
satisfying
\begin{equation}\label{Farkas-Lemma}
\bm y_J\ge\bm0,\quad\bm y_K\le\bm0,\quad\bm yU=\bm 0\;\And\; \langle\bm y,\bm b\rangle<0.
\end{equation}
From the definition of the circuit vector in \eqref{circuit-vector},  we have that each circuit vector $\bm c^C$ satisfies $\bm c^CU=\bm 0$ for $C\in\mathcal{C}(\mathcal{A}_{\bm o})$. Note that $\mathcal{C}(\mathcal{A}_{\bm o})$ is the set of  inclusion-minimum nonempty supports of members in the solution space of the equation $\bm x U=\bm 0$. Hence, there exists a circuit $C_0\in\mathcal{C}(\mathcal{A}_{\bm o})$ such that $C_0\subseteq \supp(\bm y)$ and its circuit vector $\bm c=(c_1,\ldots,c_m)$ satisfying $\bm cU=\bm 0$. Suppose $\bm c_J\ge\bm0$ and $\bm c_K\le\bm 0$ and $\langle\bm c,\bm b\rangle<0$,  then we arrive at $\langle\bm c,\bm a\rangle<0$ via $\bm b\in F_{\bm a}$. By \autoref{General-Farkas-lemma}, we obtain $P(\bm a,I,J,K)=\emptyset$, a contradiction. Therefore, we can assume that $\bm c$ coincides with one of the following both cases:
\begin{itemize}
\item $\bm c_J$ contains negative entries, or $\bm c_K$ contains positive entries and $\langle\bm c,\bm b\rangle\le0$;
\item $\bm c_J=\bm 0,\bm c_K=\bm 0$ and $\langle\bm c,\bm b\rangle=0$.
\end{itemize}
For the former case, let $\lambda=\min\big\{-\frac{y_j}{c_j},-\frac{y_k}{c_k}\mid c_j<0,j\in J\And c_k>0,k\in K\big\}$ and $\bm y'=\bm y+\lambda\bm c$. For the latter case, let $\lambda=-\frac{y_i}{c_i}$ for some $i\in C_0$ and $\bm y'=\bm y+\lambda\bm c$. Then, in either case, $\bm y'$ always satisfies
\begin{equation*}
\bm y'_J\ge\bm 0,\quad\bm y'_K\le \bm 0,\quad\bm y'U=0,\quad \langle\bm y',\bm b\rangle<0\;\And\;\supp(\bm y')\subsetneqq\supp(\bm y).
\end{equation*}
If $\bm y'$ is not a circuit vector of $\mathcal{A}_{\bm o}$,  then as with $\bm y$, we can repeat the same step to yield a row vector $\bm y''\in\mathbb{R}^m$ with the same properties as  \eqref{Farkas-Lemma} and $\supp(\bm y'')\subsetneqq\supp(\bm y')$. After finitely many steps, we eventually obtain a circuit vector $\bm y^{'''}\in\mathbb{R}^m$ satisfying \eqref{Farkas-Lemma}. This implies $\langle\bm y^{'''},\bm a\rangle<0$ from $\bm b\in F_{\bm a}$.  Using \autoref{General-Farkas-lemma} again, we arrive at $P(\bm a,I,J,K)=\emptyset$, which contradicts the assumption $P(\bm a,I,J,K)\ne\emptyset$. So we get $P(\bm b,I,J,K)\ne\emptyset$.
\end{proof}

Associated with a tetrad $(\bm a, I,J,K)$, we give a (possibly empty) relatively open convex polyhedron $\mathring{P}(\bm a,I,J,K)$ in $\mathbb{R}^n$ to be
\begin{equation*}\label{Open-Polyhedral-Set}
\mathring{P}(\bm a,I,J,K):=\big\{\bm x\in\mathbb{R}^n\mid U_I\bm x= \bm a_I,\;U_J\bm x< \bm a_J,\;U_K\bm x> \bm a_K\big\}.
\end{equation*}
Using the same method as in the proof of  \autoref{Nonempty}, we also achieve the following lemma.
\begin{lemma}\label{Nonempty-Open}
Let $\mathring{P}(\bm a,I,J,K)\ne\emptyset$ and $F_{\bm a}\in\mathcal{F}(\delta\mathcal{A}_{\bm o})$. Then $\mathring{P}(\bm b,I,J,K)\ne\emptyset$ for any $\bm b\in\relint(F_{\bm a})$.
\end{lemma}
We now proceed to describe the boundedness of convex polyhedra.
\begin{lemma}\label{Unbounded1}
Suppose $P(\bm a,I,J,K)\ne\emptyset$. Then $P(\bm a,I,J,K)$ is unbounded if and only if there exist $\bm u\in\relint\big(P(\bm a,I,J,K)\big)$ and $\bm v\in\mathbb{R}^n\setminus\{\bm 0\}$ such that $\{\bm u+r\bm v\mid r\ge0\}\subseteq P(\bm a,I,J,K)$.
\end{lemma}
\begin{proof}
For sufficiency, since $P(\bm a,I,J,K)$ contains a ray $\{\bm u+r\bm v\mid r\ge0\}$, $P(\bm a,I,J,K)$ is unbounded. Next, we verify the necessity by negation. Suppose  $P(\bm a,I,J,K)\cap\{\bm u+r\bm v\mid r\ge0\}\ne \{\bm u+r\bm v\mid r\ge0\}$ for any $\bm u\in\relint\big(P(\bm a,I,J,K)\big)$ and $\bm v\in\mathbb{R}^n\setminus\{\bm 0\}$. Since $P(\bm a,I,J,K)$ is a closed convex set, there is $r_{\bm u,\bm v}\ge0$ such that
\[
P(\bm a,I,J,K)\cap\{\bm u+r\bm v\mid r\ge0\}=[\bm u,\bm u+r_{\bm u,\bm v}\bm v],
\]
where $[\bm u,\bm u+r_{\bm u,\bm v}\bm v]$ denotes the segment from $\bm u$ to $\bm u+r_{\bm u,\bm v}\bm v$. Given $\bm u_0\in\relint\big(P(\bm a,I,J,K)\big)$, we have
\begin{align*}
  P(\bm a,I,J,K)&=P(\bm a,I,J,K)\bigcap\bigcup_{\bm v\in\mathbb{S}^{n-1}}\{\bm u_0+r\bm v\mid r\ge0\}\\
  &=P(\bm a,I,J,K)\bigcap\bigcup_{\bm v\in\mathbb{S}^{n-1}}[\bm u_0,\bm u_0+r_{\bm u_0}\bm v]\\
  &\subseteq B^n(\bm u_0,r_{\bm u_0}),
\end{align*}
where $\mathbb{S}^{n-1}:=\{\bm x\in\mathbb{R}^n\mid ||\bm x||=1\}$ is the unit sphere, and $B^n(\bm u_0,r_{\bm u_0})$ is the ball with centre $\bm u_0$ and radius $r_{\bm u_0}$ given by
\[
r_{\bm u_0}=\sup_{\bm v\in\mathbb{S}^{n-1}}\big\{r_{\bm u_0,\bm v}\mid P(\bm a,I,J,K)\cap\{\bm u_0+r\bm v\mid r\ge0\}=[\bm u_0, \bm u_0+r_{\bm u_0,\bm v}\bm v]\big\}.
\]
This means that $P(\bm a,I,J,K)$ is bounded, a contradiction. So $P(\bm a,I,J,K)$ is unbounded. 
\end{proof}

\begin{lemma}\label{Unbounded2}
Suppose $P(\bm a,I,J,K)\ne\emptyset$. Then there exist $\bm u\in\relint\big(P(\bm a,I,J,K)\big)$ and $\bm v\in\mathbb{R}^n\setminus\{\bm 0\}$ such that $\{\bm u+r\bm v\mid r\ge0\}\subseteq P(\bm a,I,J,K)$ if and only if there is $\bm v\in\mathbb{R}^n\setminus\{\bm 0\}$ such that $\langle\bm u_i,\bm v\rangle=0$, $\langle\bm u_j,\bm v\rangle\le 0$ and $\langle\bm u_k,\bm v\rangle\ge 0$ for each $i\in I$, $j\in J$ and $k\in K$.
\end{lemma}
\begin{proof}
Without loss of generality, suppose $I=\big\{i\in[m]\mid P(\bm a,I,J,K)\subseteq H_{\bm u_i,a_i}\big\}$. For necessity, we first assert that $\bm u+r\bm v$ belongs to $\mathring{H}^-_{\bm u_j,a_j}$ ($\mathring{H}^+_{\bm u_k,a_k}$, resp.) for any $r\ge 0$ and $j\in J$ ($k\in K$, resp.).  Arguing by negation, suppose there are $r_0>0$ and $j_0\in J$ satisfying $\bm u+r_0\bm v\in H_{\bm u_{j_0},a_{j_0}}$, i.e., $\langle\bm u_{j_0},\bm u+r_0\bm v\rangle=a_{j_0}$.
Noting $\langle\bm u_{j_0},\bm u\rangle<a_{j_0}$ and $r_0>0$, then
\[
\langle\bm u_{j_0},\bm v\rangle=\frac{a_{j_0}-\langle\bm u_{j_0},\bm u\rangle}{r_0}>0.
\]
Hence, for any $\varepsilon>0$, we have $\langle\bm u_{j_0},\bm u+(r_0+\varepsilon)\bm v\rangle>a_{j_0}$. Namely, $\bm u+(r_0+\varepsilon)\bm v\notin P(\bm a,I,J,K)$ for all $\varepsilon>0$, which contradicts $\bm u+r\bm v\in P(\bm a,I,J,K)$ for $r\ge0$. So, $\bm u+r\bm v\in \mathring{H}^-_{\bm u_j,a_j}$ for $r\ge0$ and $j\in J$. Using the same argument as the former case, we can show $\bm u+r\bm v\in \mathring{H}^+_{\bm u_k,a_k}$ for $r\ge0$ and $k\in K$ as well.
Therefore, the assertion holds. In other words, the assertion says that
\[
\langle\bm u_j,\bm u+r\bm v\rangle<a_j\quad\And\quad\langle\bm u_k,\bm u+r\bm v\rangle>a_k
\]
for $r\ge 0$, $j\in J$ and $k\in K$. This implies $\langle\bm u_j,\bm v\rangle\le0$ and $\langle\bm u_k,\bm v\rangle\ge0$ for $j\in J$ and $k\in K$. In addition, for $i\in I$ and $r\ge 0$, we have $\langle\bm u_i,\bm u\rangle=a_i$ and $\langle\bm u_i,\bm u+r\bm v\rangle=a_i$. So $\langle\bm u_i,\bm v\rangle=0$ for any $i\in I$.

For sufficiency, taking $\bm u\in\relint\big(P(\bm a,I,J,K)\big)$, clearly $\langle\bm u_i,\bm u\rangle=a_i$, $\langle\bm u_j,\bm u\rangle\le a_j$ and  $\langle\bm u_k,\bm u\rangle\ge a_k$ for each $i\in I$, $j\in J$ and $k\in K$. Together with conditions $\langle\bm u_i,\bm v\rangle=0$, $\langle\bm u_j,\bm v\rangle\le0$ and  $\langle\bm u_k,\bm v\rangle\ge 0$ for $i\in I$, $j\in J$ and $k\in K$, we have
\[
\langle\bm u_i,\bm u+r\bm v\rangle=a_i,\quad\langle\bm u_j,\bm u+r\bm v\rangle\le a_j\quad\And\quad  \langle\bm u_k,\bm u+r\bm v\rangle\ge a_k
\]
for any $i\in I$, $j\in J$, $k\in K$ and $r\ge 0$. Namely, the ray $\{\bm u+r\bm v\mid r\ge0\}$ is contained in $P(\bm a,I,J,K)$. This finishes the proof.
\end{proof}

The next theorem states that the convex polyhedra $P(\bm a,I,J,K)$ and $P(\bm b,I,J,K)$ are either both bounded or unbounded whenever $\bm a$ and $\bm b$ come from the same face of $\delta\mathcal{A}_{\bm o}$.
\begin{theorem}\label{Bounded-Unbounded}
Let $F\in\mathcal{F}(\delta\mathcal{A}_{\bm o})$ and $\bm a, \bm b\in F$. For any partition $I,J,K$ of $[m]$,  $P(\bm a,I,J,K)$ is a polytope {\rm(}an unbounded convex polyhedron, resp.{\rm)} if and only if $P(\bm b,I,J,K)$ is a polytope {\rm(}an unbounded convex polyhedron, resp.{\rm)}.
\end{theorem}
\begin{proof}
From \autoref{Nonempty}, we have that $P(\bm a,I,J,K)\ne\emptyset$ if and only if $P(\bm b,I,J,K)\ne\emptyset$. Suppose $P(\bm a,I,J,K)\ne\emptyset$. Summarizing \autoref{Unbounded1} and \autoref{Unbounded2},  we acquire that  $P(\bm a,I,J,K)$ is bounded (unbounded, resp.) if and only if $P(\bm b,I,J,K)$ is bounded (unbounded, resp.). This completes the proof.
\end{proof}

For each partition $I,J,K$ of $[m]$, we denote by $\mathcal{C}(I,J,K)$ the collection of all vectors $\bm a$ satisfying $P(\bm a,I,J,K)\ne\emptyset$, i.e.,
\[
\mathcal{C}(I,J,K):=\big\{\bm a\in\mathbb{R}^m\mid P(\bm a,I,J,K)\ne\emptyset\big\}.
\]
Fillastre and Izmestiev in \cite[Theorem 1.30]{Fillastre-Izmestiev2017} showed that the solution set $\mathcal{C}(I,J,K)$ is always a polyhedral cone in the parameter space $\mathbb{R}^m$. This leads us to consider the following result.
\begin{corollary}\label{Cone}
For any partition $I,J,K$ of $[m]$, the boundedness of  the convex polyhedra  $P(\bm a,I,J,K)$  keeps consistent for all $\bm a\in \mathcal{C}(I,J,K)$.
\end{corollary}
\begin{proof}
When $\mathcal{C}(I,J,K)=\emptyset$, the case is trivial. Suppose $\mathcal{C}(I,J,K)$ is nonempty. From  \autoref{Nonempty}, we know that if $\bm x\in \mathcal{C}(I,J,K)$, then $P(\bm y,I,J,K)$ is nonempty for any $\bm y\in F_{\bm x}$ with $F_{\bm x}\in\mathcal{F}(\delta\mathcal{A}_{\bm o})$. Namely, $F_{\bm x}\subseteq \mathcal{C}(I,J,K)$ for each $\bm x\in\mathcal{C}(I,J,K)$. Combining  \cite[Theorem 1.30]{Fillastre-Izmestiev2017}, this implies that $\mathcal{C}(I,J,K)$ is a polyhedral cone consisting of finitely many polyhedral cones like $F_{\bm x}$. Together with \autoref{Bounded-Unbounded}, we deduce that the convex polyhedra $P(\bm b,I,J,K)$ are bounded (unbounded, resp.) for all $\bm b\in\mathcal{C}(I,J,K)$ whenever $P(\bm a,I,J,K)$ is bounded (unbounded, resp.) for some $\bm a\in\mathcal{C}(I,J,K)$. We complete the proof.
\end{proof}
\subsection{Proof of \autoref{Normally-Combinatorially}}\label{Sec-3-2}
To prove \autoref{Normally-Combinatorially}, some basic properties of faces of convex polyhedra are required.
\begin{lemma}[\cite{Ewald1996}]\label{Face-Face}
The faces of a convex polyhedron $P$ have the following properties:
\begin{itemize}
\item [\i] If  $F_1,\ldots, F_k$ are the faces of $P$, then the intersection $F_1\cap \cdots\cap F_k$ is a face of $P$.
\item [\ii] If $F$ is a face of $P$, then $F$ is the intersection of all facets of $P$ containing $F$.
\end{itemize}
\end{lemma}
\begin{lemma}\label{Facet-Characterization}
Suppose $P(\bm a, I,J,K)\ne\emptyset$. If $F$ is a facet of $P(\bm a, I,J,K)$,  then $F=H_{\bm u_j,a_j}\cap P(\bm a, I,J,K)$ for some $j\in J\cup K$.
\end{lemma}
\begin{proof}
Without loss of generality, suppose $I=\emptyset$ and $P(\bm a,\emptyset,J,K)$ is a full-dimensional convex polyhedron in $\mathbb{R}^n$. Let $\bm x$ be a relative interior of the facet $F$. Note the basic fact that $P(\bm a,\emptyset,J,K)$ is the disjoint union of the relative interiors of its faces. Then $\bm x$ is on the boundary of $P(\bm a,\emptyset,J,K)$. Since the boundary of $P(\bm a,\emptyset,J,K)$ is equal to $\bigcup_{i\in[m]}\big(H_{\bm u_i,a_i}\cap P(\bm a,\emptyset,J,K)\big)$, there is $j\in[m]$ satisfying $\bm x\in H_{\bm u_j,a_j}$. Assume $j\in J$. It follows from $P(\bm a,\emptyset,J,K)\subseteq H_{\bm u_j,a_j}^-$ that $H_{\bm u_j,a_j}\cap P(\bm a,\emptyset,J,K)$ is a face of $P(\bm a,\emptyset,J,K)$. Since $F$ is the smallest face of $P(\bm a,\emptyset,J,K)$ containing $\bm x$, $F\subseteq H_{\bm u_j,a_j}\cap P(\bm a,\emptyset,J,K)$. Then we have
\[
n-1=\dim(F)\le\dim\big(H_{\bm u_j,a_j}\cap P(\bm a,\emptyset,J,K)\big)\le\dim(H_{\bm u_j,a_j})=n-1.
\]
This means that $H_{\bm u_j,a_j}\cap P(\bm a,\emptyset,J,K)$ is a facet of $P(\bm a,\emptyset,J,K)$. So $F=H_{\bm u_j,a_j}\cap P(\bm a,\emptyset,J,K)$.
\end{proof}
Suppose $P(\bm a,I,J,K)\ne\emptyset$. For each $\bm x\in P(\bm a,I,J,K)$,  the {\em active index set} is given by
\[
I_{P(\bm a,I,J,K)}(\bm x):=\big\{i\in[m]\mid \bm x\in H_{\bm u_i,a_i}\big\}.
\]
If $\bm x$ and $\bm y$ are two relative interior points of a face $F$ of  $P(\bm a,I,J,K)$, then $I_{P(\bm a,I,J,K)}(\bm x)=I_{P(\bm a,I,J,K)}(\bm y)$ obviously. For this reason, we define the {\em active index set} $I_{P(\bm a,I,J,K)}(F)$ of each face $F\in\mathcal{F}\big(P(\bm a,I,J,K)\big)$ as
\begin{equation*}\label{Active-Index-Set}
I_{P(\bm a,I,J,K)}(F):=I_{P(\bm a,I,J,K)}(\bm x),\quad\forall\; \bm x\in\relint(F),
\end{equation*}
and denote by
\[
J_{P(\bm a,I,J,K)}(F):=J\setminus I_{P(\bm a,I,J,K)}(F)\quad\And\quad K_{P(\bm a,I,J,K)}(F):=K\setminus I_{P(\bm a,I,J,K)}(F).
\]
We call $\big(I_{P(\bm a,I,J,K)}(F),J_{P(\bm a,I,J,K)}(F),K_{P(\bm a,I,J,K)}(F)\big)$ an {\em active triple index} of the face $F$ induced by $P(\bm a,I,J,K)$. The following lemma shows that the active triple index completely characterizes a nonempty face $F$ of $P(\bm a,I,J,K)$, and also determines the dimension of $F$.

\begin{lemma}\label{Face-Characterization}
Let $P(\bm a,I,J,K)\ne\emptyset$ and $F$ be a nonempty face of $P(\bm a,I,J,K)$. Then we have
\[
F=P\big(\bm a,I_{P(\bm a,I,J,K)}(F),J_{P(\bm a,I,J,K)}(F),K_{P(\bm a,I,J,K)}(F)\big)
\]
and
\[
\dim(F)=n-\rank\big\{\bm u_i\mid i\in I_{P(\bm a,I,J,K)}(F)\big\}.
\]
\end{lemma}
\begin{proof}
Without loss of generality, suppose $I=\emptyset$ and $P(\bm a,\emptyset,J,K)$ is a full-dimensional convex polyhedron in $\mathbb{R}^n$. When $F=P$, the case is trivial. Next we consider that $F$ is a proper face. Since $F\subseteq H_{\bm u_i,a_i}$ for each $i\in I_{P(\bm a,\emptyset,J,K)}(F)$ and $F\subseteq P(\bm a,\emptyset,J,K)$,  we have
\[
F\subseteq \bigcap_{i\in I_{P(\bm a,\emptyset,J,K)}(F)}H_{\bm u_i,a_i}\bigcap P(\bm a,\emptyset,J,K).
\]
Then $H_{\bm u_i,a_i}\cap P(\bm a,\emptyset,J,K)\ne\emptyset$ for $i\in I_{P(\bm a,\emptyset,J,K)}(F)$. So $H_{\bm u_i,a_i}\cap P(\bm a,\emptyset,J,K)$ is a proper face of $P(\bm a,\emptyset,J,K)$ for each $i\in I_{P(\bm a,\emptyset,J,K)}(F)$. Let
\[
I_0=\big\{i\in I_{P(\bm a,\emptyset,J,K)}(F)\mid \dim\big(H_{\bm u_i,a_i}\cap P(\bm a,\emptyset,J,K)\big)=n-1 \big\}.
\]
According to  \autoref{Facet-Characterization}, we know that the set of all facets of $P(\bm a,\emptyset,J,K)$ containing $F$ is exactly the set $\big\{H_{\bm u_i,a_i}\cap P(\bm a,\emptyset,J,K)\mid i\in I_0\big\}$. This is a subset of $\big\{H_{\bm u_i,a_i}\cap P(\bm a,\emptyset,J,K)\mid i\in I_{P(\bm a,\emptyset,J,K)}(F)\big\}$.  Recall from the property (ii) in \autoref{Face-Face} that we have
\[
\bigcap_{i\in I_{P(\bm a,\emptyset,J,K)}(F)}H_{\bm u_i,a_i}\bigcap P(\bm a,\emptyset,J,K)\subseteq \bigcap_{i\in I_0}\big(H_{\bm u_i,a_i}\cap P(\bm a,\emptyset,J,K)\big)=F.
\]
So $F=\bigcap_{i\in I_{P(\bm a,\emptyset,J,K)}(F)}H_{\bm u_i,a_i}\bigcap P(\bm a,\emptyset,J,K)$. This implies that $F$ can be written in the form
\[
F=P\big(\bm a,I_{P(\bm a,\emptyset,J,K)}(F),J_{P(\bm a,\emptyset,J,K)}(F),K_{P(\bm a,\emptyset,J,K)}(F)\big).
\]
Since $\big(\bigcap_{j\in J_{P(\bm a,\emptyset,J,K)}(F)}\mathring{H}_{\bm u_j,a_j}^-\big)\bigcap\big(\bigcap_{i\in K_{P(\bm a,\emptyset,J,K)}(F)}\mathring{H}_{\bm u_i,a_i}^+\big)$ is an $n$-dimensional open convex polyhedron, we arrive at $\dim(F)=\dim\big(\bigcap_{i\in I_{P(\bm a,\emptyset,J,K)}(F)}H_{\bm u_i,a_i}\big)$. Immediately, we have
\begin{equation*}
\dim(F)=n-\rank\big\{\bm u_i\mid i\in I_{P(\bm a,\emptyset,J,K)}(F)\big\}
\end{equation*}
via the elementary linear algebra. This completes the proof.
\end{proof}

When vectors $\bm a$ and $\bm b$ are from the same open face of the derived arrangement $\delta\mathcal{A}_{\bm o}$, we can establish a one-to-one correspondence between the faces of $P(\bm a,I,J,K)$ and $P(\bm b,I,J,K)$.  This is a key property to classify the normally equivalent convex polyhedra.
\begin{lemma}\label{Face-Bijection}
Let $F\in\mathcal{F}(\delta\mathcal{A}_{\bm o})$ and $\bm a,\bm b\in\relint(F)$. Then the map $\Phi: \mathcal{F}\big(P(\bm a,I,J,K)\big)\to\mathcal{F}\big(P(\bm b,I,J,K)\big)$ defined as $\Phi(\emptyset)=\emptyset$ and
 \[
G\mapsto\Phi(G)=P\big(\bm b,I_{P(\bm a,I,J,K)}(G),J_{P(\bm a,I,J,K)}(G),K_{P(\bm a,I,J,K)}(G)\big),\;\forall G\in\mathcal{F}\big(P(\bm a,I,J,K)\big)\setminus\{\emptyset\},
\]
is a bijection. Moreover, $G$ and $\Phi(G)$ have the same active triple index for each nonempty face $G$ of  $P(\bm a,I,J,K)$.
\end{lemma}
\begin{proof}
According to \autoref{Nonempty}, $P(\bm a,I,J,K)$ and $P(\bm b,I,J,K)$ are either both empty or nonempty sets. When $P(\bm a,I,J,K)$ and $P(\bm b,I,J,K)$ are empty sets,  there is nothing to prove.  Suppose $P(\bm a,I,J,K)\ne\emptyset$. In this case, we will prove that $\Phi$ is well-defined, injective and surjective in turn. Clearly we have  $\Phi(G)=\bigcap_{i\in  I_{P(\bm a,I,J,K)}(G)}\big(H_{\bm u_i,b_i}\cap P(\bm b,I,J,K)\big)$. Since $H_{\bm u_i,b_i}\cap P(\bm b,I,J,K)$ is a face of $P(\bm b,I,J,K)$ for each $i\in  I_{P(\bm a,I,J,K)}(G)$,  $\Phi(G)$ is a face of $P(\bm b,I,J,K)$ via \i in \autoref{Face-Face}. Namely, $\Phi$ is well-defined.

To verify the injectivity of $\Phi$, we first show that $\Phi(G)\ne\emptyset$ whenever $G$ is a nonempty face of $P(\bm a,I,J,K)$. From \autoref{Face-Characterization}, a nonempty face $G$ of $P(\bm a,I,J,K)$ can be expressed as $G=P\big(\bm a,I_{P(\bm a,I,J,K)}(G),J_{P(\bm a,I,J,K)}(G),K_{P(\bm a,I,J,K)}(G)\big)$. Using \autoref{Nonempty}, we obtain $\Phi(G)\ne\emptyset$. Before proceeding further, we  need to prove
\begin{equation}\label{A=B}
I_{P(\bm b,I,J,K)}\big(\Phi(G)\big)=I_{P(\bm a,I,J,K)}(G).
\end{equation}
Noting $I_{P(\bm a,I,J,K)}(G)\subseteq I_{P(\bm b,I,J,K)}\big(\Phi(G)\big)$, we show the opposite inclusion by negation. Suppose there is $i_0\in I_{P(\bm b,I,J,K)}\big(\Phi(G)\big)\setminus I_{P(\bm a,I,J,K)}(G)$. Combining \autoref{Face-Characterization}, we deduce
\[
\dim(\Phi(G))\le\dim(G)-1.
\]
Since $\relint(G)=\mathring{P}\big(\bm a,I_{P(\bm a,I,J,K)}(G),J_{P(\bm a,I,J,K)}(G),K_{P(\bm a,I,J,K)}(G)\big)$ is nonempty set, we have $\mathring{P}\big(\bm b,I_{P(\bm a,I,J,K)}(G),J_{P(\bm a,I,J,K)}(G),K_{P(\bm a,I,J,K)}(G)\big)$ is nonempty by \autoref{Nonempty-Open}.  Based on the fundamental topological theory, we further obtain
\begin{align*}
\dim\big(\Phi(G)\big)&\ge\dim\big(\mathring{P}\big(\bm b,I_{P(\bm a,I,J,K)}(G),J_{P(\bm a,I,J,K)}(G),K_{P(\bm a,I,J,K)}(G)\big)\big)\\
&=\dim\Big(\bigcap_{i\in I_{P(\bm a,I,J,K)}(G)}\Big)H_{\bm u_i,b_i}\\
&=n-\rank\big\{\bm u_i\mid i\in I_{P(\bm a,I,J,K)}(G)\big\}.
\end{align*}
Together with \autoref{Face-Characterization}, this means
\[
\dim\big(\Phi(G)\big)\ge\dim(G),
\]
a contradiction. So we have proved \eqref{A=B}. Taking another nonempty face $G'$ of $P(\bm a,I,J,K)$, obviously $I_{P(\bm a,I,J,K)}(G)\ne I_{P(\bm a,I,J,K)}(G')$. It follows from \eqref{A=B} that $\Phi$ is an injection.

For the surjectivity, let us first define the map $\Phi^{-1}$ as $\Phi^{-1}(\emptyset)=\emptyset$ and
 \[
\Phi^{-1}(G)=P\big(\bm a,I_{P(\bm b,I,J,K)}(G),J_{P(\bm b,I,J,K)}(G),K_{P(\bm b,I,J,K)}(G)\big),\;\forall\, G\in\mathcal{F}\big(P(\bm b,I,J,K)\big)\setminus\{\emptyset\}.
\]
Using the same argument as the map $\Phi$, we can show that $\Phi^{-1}$ is injective and
\[
I_{P(\bm a,I,J,K)}\big(\Phi^{-1}(G)\big)=I_{P(\bm b,I,J,K)}(G),\quad\forall\,  G\in\mathcal{F}\big(P(\bm b,I,J,K)\big)\setminus\{\emptyset\}.
\]
Together with \eqref{A=B}, we have $\Phi\big(\Phi^{-1}(G)\big)=G$ for any face $G$ of $P(\bm b,I,J,K)$. Hence, $\Phi$ is surjective. Up to now we have shown that $\Phi$ is a bijection.

Moreover, \eqref{A=B}  indicates that $G$ and $\Phi(G)$ have the same active triple index from the definition of the active triple index. We complete the proof.
\end{proof}

It is worth remarking that the polyhedral cone can be expressed as the {\em positive hull}
\[
\pos\{\bm u_1,\ldots,\bm u_m\}:=\Big\{\sum_{i=1}^{m}\lambda_i\bm u_i\mid\lambda_i\ge 0\Big\}.
\]
By convention, define $\pos\{\emptyset\}=\bm o$. Generally, we denote by $\Span\{\bm u_1,\ldots,\bm u_m\}$ the vector space consisting of all linear combinations of these vectors $\bm u_1,\ldots,\bm u_m$. The following is a formula to calculate the normal cones of convex polyhedra, which can be founded in \cite[6.46 Theorem]{Rockafellar1998}.
\begin{lemma}[\cite{Rockafellar1998}, 6.46 Theorem]\label{Normal-Cone1}
Suppose $P(\bm a,I,J,K)\ne \emptyset$. Then, for any nonempty face $F$ of $P(\bm a,I,J,K)$,
\[
N_{P(\bm a,I,J,K)}(F)=\pos\big\{\bm u_j,-\bm u_k\mid  j\in I_{P(\bm a,I,J,K)}(F)\cap J,k\in I_{P(\bm a,I,J,K)}(F)\cap K\big\}+\Span\big\{\bm u_i\mid i\in I\big\}.
\]
\end{lemma}

In 2000, Kim and Luc \cite{Kim-Luc2000} builded the relationship between faces and  normal cones of  convex polyhedra.
\begin{lemma}[\cite{Kim-Luc2000}, Proposition 3.4]\label{Normal-Polyhedral-Face}
Let $P(\bm a,I,J,K)\ne\emptyset$, and $F_1,F_2$ be nonempty faces of $P(\bm a,I,J,K)$. Then $F_1$ is a face of $F_2$ if and only if $N_{P(\bm a,I,J,K)}(F_2)$ is a face of $N_{P(\bm a,I,J,K)}(F_1)$.
\end{lemma}

The next proposition tells us that the normal equivalence is a stronger relation than the combinatorial equivalence for convex polyhedra.
\begin{proposition}\label{Both-Equivalent-Relationship}
If two convex polyhedra $P$ and $Q$ in $\mathbb{R}^n$ are normally equivalent, then they are combinatorially equivalent.
\end{proposition}
\begin{proof}
According to the definition of the normally equivalent polyhedra in \autoref{Two-Equivalence-Polyhedra}, there is a natural bijection $\Phi:\mathcal{F}(P)\to\mathcal{F}(Q)$ satisfying $\Phi(\emptyset)=\emptyset$
and $N_P(F)=N_Q\big(\Phi(F)\big)$ for any nonempty face $F$ of $P$. Given faces $F,G\in\mathcal{F}(P)\setminus\{\emptyset\}$ with $F\subseteq G$, it follows from \autoref{Normal-Polyhedral-Face} that $\Phi(F)\subseteq\Phi(G)$. Hence, $P$ and $Q$ are combinatorially equivalent.
\end{proof}
Conversely, although convex polyhedra are combinatorially equivalent, they are not necessarily normally equivalent. For example, let $P=\big\{(x,y)\in\mathbb{R}^2\mid x\ge 0,\,y\ge 0\big\}$ and $Q=\big\{(x,y)\in\mathbb{R}^2\mid x\ge 0,\,x-y\ge 0\big\}$ be convex polyhedra in $\mathbb{R}^2$. Clearly $\mathcal{F}(P)\cong\mathcal{F}(Q)$, but $P$ and $Q$ are not normally equivalent. We now return to verifying \autoref{Normally-Combinatorially}.
\begin{proof}[Proof of \autoref{Normally-Combinatorially}]
Applying \autoref{Nonempty}, $P(\bm a,I,J,K)$ and $P(\bm b,I,J,K)$ are either both empty or nonempty sets. When $P(\bm a,I,J,K)$ and $P(\bm b,I,J,K)$ are empty sets, the case is trivial. Suppose $P(\bm a,I,J,K)\ne\emptyset$. According to the bijection $\Phi: \mathcal{F}\big(P(\bm a,I,J,K)\big)\to\mathcal{F}\big(P(\bm b,I,J,K)\big)$ in \autoref{Face-Bijection}, we know that for each nonempty face $G$ of $P(\bm a,I,J,K)$, $G$ and $\Phi(G)$ have the same active index set. Together with \autoref{Normal-Cone1}, we further arrive at that
$P(\bm a,I,J,K)$ and $P(\bm b,I,J,K)$ have the same normal fan. Namely, $P(\bm a,I,J,K)$ and $P(\bm b,I,J,K)$ are normally equivalent. With \autoref{Both-Equivalent-Relationship}, we directly deduce that $P(\bm a,I,J,K)$ and $P(\bm b,I,J,K)$ are combinatorially equivalent.
\end{proof}

At the end of the section, let us quickly look at a small example of the extended deformations of convex polyhedra.
\begin{example}\label{Example1}
{\rm
Let $P(\bm a,I,J,K)$ be a polytope in $\mathbb{R}^2$ given by the following inequality systems, where $\bm a=(0,1,0,1)$, $I=K=\emptyset$ and $J=[4]$.
\begin{figure}[H]
\centering
\begin{tikzpicture}[scale=1,line width=1pt]
\draw (-6,0.75) node[black]
{$
\bordermatrix{
              &     &\cr
\bm u_1 & -1 &0 \cr
\bm u_2 & 0 & 1  \cr
\bm u_3 &0 & -1 \cr
\bm u_4 & 1 & 1
}
\bordermatrix{
& \cr & x\cr & y
}\le\bordermatrix{
& \cr & 0\cr & 1\cr & 0\cr &1
}
$};
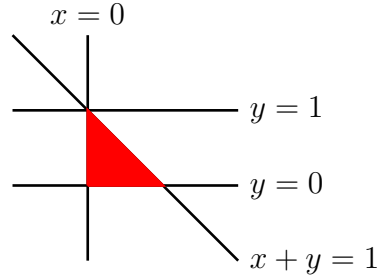
\draw (-6,-1.75) node[black] {$P(\bm a,I,J,K)$};

\draw [black,-](-1,0) -- (2,0) node [right, black] {$y=0$};
\draw [black,-](-1,1) -- (2,1) node [right, black] {$y=1$};
\draw [black,-](0,-1) -- (0,2) node [above, black] {$x=0$};
\draw [black,-](-1,2) -- (2,-1) node [right, black] {$x+y=1$};
\filldraw[red] (0,0)--(1,0)--(0,1)--(0,0);
\draw (0.75,-1.75) node[black] {$\bm a=(0,1,0,1)$};
\end{tikzpicture}
\end{figure}

Three distinct extended deformations of $P(\bm a,I,J,K)$ are illustrated as red faces below.
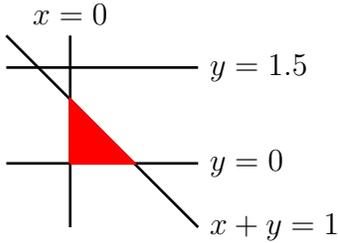
\begin{figure}[H]
\centering
\begin{tikzpicture}[scale=0.85,line width=1pt]
\draw [black,-](-1,0) -- (2,0) node [right, black] {$y=0$};
\draw [black,-](-1,1.5) -- (2,1.5) node [right, black] {$y=1.5$};
\draw [black,-](0,-1) -- (0,2) node [above, black] {$x=0$};
\draw [black,-](-1,2) -- (2,-1) node [right, black] {$x+y=1$};
\draw (0.75,-1.75) node[black] {$\bm b_1=(0,1.5,0,1)$};

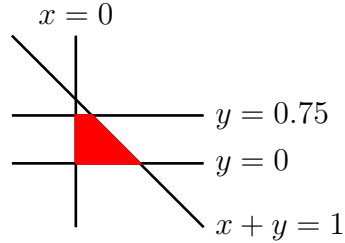
\filldraw[red] (0,0)--(1,0)--(0,1)--(0,0);
\draw [black,-](5,0) -- (8,0) node [right, black] {$y=0$};
\draw [black,-](5,0.75) -- (8,0.75) node [right, black] {$y=0.75$};
\draw [black,-](6,-1) -- (6,2) node [above, black] {$x=0$};
\draw [black,-](5,2) -- (8,-1) node [right, black] {$x+y=1$};
\filldraw[red] (6,0)--(7,0)--(6.25,0.75)--(6,0.75)--(6,0);
\draw (6.75,-1.75) node[black] {$\bm b_2=(0,0.75,0,1)$};

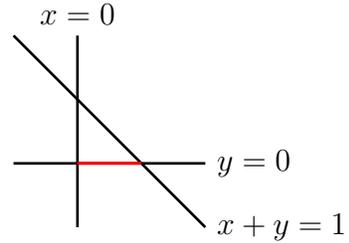
\draw [black,-](11,0) -- (14,0) node [right, black] {$y=0$};
\draw [black,-](12,-1) -- (12,2) node [above, black] {$x=0$};
\draw [black,-](11,2) -- (14,-1) node [right, black] {$x+y=1$};
\draw (12.75,-1.75) node[black] {$\bm b_3=(0,0,0,1)$};
\filldraw[red](12,0)--(13,0)--(12,0);
\end{tikzpicture}
\end{figure}
It is easy to see that each extended deformation of $P(\bm a,I,J,K)$ is associated with at least one extended deforming vector. Namely, the configuration space of extended deformations of $P(\bm a,I,J,K)$ can be parameterized by continuous multi-variables. However, this way is not optimal. For example, although $P(\bm b_1,I,J,K)$ is precisely the same polytope as $P(\bm a,I,J,K)$, the corresponding extended deforming vectors $\bm b_1$ and $\bm a$ are not equal. On the other hand, it is clear that the normal fan $\mathcal{N}\big(P(\bm a,I,J,K)\big)$ coarsens the  normal fan $\mathcal{N}\big(P(\bm b_2,I,J,K)\big)$ and refines the normal fan $\mathcal{N}\big(P(\bm b_3,I,J,K)\big)$.  This further implies that the two definitions of extended deformation of convex polyhedra in \autoref{Deformation-Polyhedra} and \cite[Definition 2.3]{Ardila2020}  are not completely equivalent.
}
\end{example}

\section{Equivalence relations and classifications on real hyperplane arrangements}\label{Sec-4}
In the section, we are concerned with two subjects. One of them is to establish explicit relationships among four types of equivalence relations on real hyperplane arrangements: semi-lattice, combinatorial, normal and sign equivalences. The other one aims to describe the equivalence classes of the latter three relations for three types of deformations of the linear arrangement $\mathcal{A}_{\bm o}$: parallel translation $\mathcal{A}_{\bm a}$, coning $c\mathcal{A}_{\bm a}$, and elementary lift $\mathcal{A}^{\bm a}$.
\subsection{Equivalence relations on real hyperplane arrangements}\label{Sec-4-1}
Let us first  introduce the coning operator of hyperplane arrangements. It allows for comparing affine and linear arrangements, and also serves as a bridge between the signs of affine arrangements and their conings. Associated with an affine hyperplane $H:u_1x_1+\cdots+u_nx_n=a$ in $\mathbb{R}^n$, we define a linear hyperplane $ c H$ in $\mathbb{R}^{n+1}$ to be $ cH:u_1x_1+\cdots+u_nx_n+ax_{n+1}=0$. Let $\mathcal{A}$ be an affine arrangement in $\mathbb{R}^n$. The {\em coning} $ c\mathcal{A}$ of $\mathcal{A}$ is a linear arrangement in $\mathbb{R}^{n+1}$ consisting of all $ cH$ with $H\in\mathcal{A}$ and an additional hyperplane $K_0:x_{n+1}=0$.  In other words, the coning of $\mathcal{A}$ is the linear arrangement $c\mathcal{A}$ in $\mathbb{R}^{n+1}$ consisting of linear spans of affine hyperplanes in $\mathcal{A}$ and a new linear arrangement $K_0$ that is parallel to $\mathbb{R}^n\times 1$.

\cite[Theorem 3.2]{Wachs1986} showed that a poset $P$ is a geometric semi-lattice if and only if there is a geometric lattice $L$ with an atom $a$ such that $P= L-\{x\in L\mid a\le x\}$, where both $L$ and $a$ are uniquely determined by $P$. It is clear that $K_0$ is exactly an atom of $L(c\mathcal{A})$ such that
\begin{equation*}
L(\mathcal{A})\cong L(c\mathcal{A})-\{X\in L(c\mathcal{A}):X\subseteq K_0\}.
\end{equation*}
Therefore,  \cite[Theorem 3.2]{Wachs1986} directly deduces the following relation
\begin{equation}\label{EQ2}
L(\mathcal{A})\cong L(\mathcal{A}')\Longleftrightarrow L(c\mathcal{A})\cong L(c\mathcal{A}')
\end{equation}
for two real affine arrangements $\mathcal{A}$ and $\mathcal{A}'$. It is also not difficult to check that
\begin{equation}\label{EQ3}
\mathcal{F}(\mathcal{A})\cong\mathcal{F}(\mathcal{A}')\Longleftrightarrow\mathcal{F}(c\mathcal{A})\cong\mathcal{F}(c\mathcal{A}').
\end{equation}
When $\mathcal{A}$ and $\mathcal{A}'$ are real linear arrangements,  \cite[Corollary 1.69]{Aguiar-Mahajan2017} stated that if $\mathcal{F}(\mathcal{A})\cong\mathcal{F}(\mathcal{A}')$,  then $L(\mathcal{A})\cong L(\mathcal{A}')$. Together with \eqref{EQ2} and \eqref{EQ3}, we can immediately extend the result to real affine arrangements. But, the converse may be not true.
\begin{proposition}\label{Combinatorial-Face}
Let $\mathcal{A}$ {\rm(}$\mathcal{A}'$, resp.{\rm)} be a hyperplane arrangement in $\mathbb{R}^n$ {\rm(}$\mathbb{R}^{n'}$, resp.{\rm)}. If $\mathcal{A}$ and $\mathcal{A}'$ are combinatorially equivalent, then they are semi-lattice equivalent.
\end{proposition}

Let $\mathcal{A}=\{H_1,\ldots,H_m\}$ be a hyperplane arrangement in $\mathbb{R}^n$. For each face $F$ of $\mathcal{A}$, we define an {\em active index set} $I_{\mathcal{A}}(F)$ as
\[
I_{\mathcal{A}}(F):=\big\{i\in[m]\mid F\subseteq H_i\big\},
\]
and denote by
\[
J_{\mathcal{A}}(F):=\big\{j\in[m]\mid F\subseteq H^-_j, F\nsubseteq H_j\big\}\quad\And\quad K_{\mathcal{A}}(F):=\big\{k\in[m]\mid F\subseteq H^+_k, F\nsubseteq H_k\big\}.
\]
The triple $\big(I_{\mathcal{A}}( F),J_{\mathcal{A}}( F),K_{\mathcal{A}}( F)\big)$ is referred to as an {\em active triple index} of the face $F$ induced by $\mathcal{A}$. In this notation, there is a canonical way to write the face $F$ in the following form
\begin{equation}\label{Face-Form}
F=P\big(\bm a, I_{\mathcal{A}}(F), J_{\mathcal{A}}(F), K_{\mathcal{A}}(F)\big).
\end{equation}

Actually, the sign vector $\Sign_{\mathcal{A}}(F)$ of a face $F$ of $\mathcal{A}$, as defined in \eqref{Face-Point-Sign},  is also characterized by its active triple index. More precisely,  $\Sign_{\mathcal{A}}(F)$ can be written in the following form  $\Sign_{\mathcal{A}}(F)=\big(\Sign_{\mathcal{A}}(F)_1,\ldots,\Sign_{\mathcal{A}}(F)_m\big)$,  where
each $i$-th component is
\begin{equation}\label{Face-Sign-Vector}
\Sign_{\mathcal{A}}(F)_i=
\begin{cases}
   0, & \mbox{if }\quad i\in I_{\mathcal{A}}(F), \\
   -, & \mbox{if } \quad i\in J_{\mathcal{A}}(F),\\
   +, & \mbox{if }\quad i\in K_{\mathcal{A}}(F).
\end{cases}
\end{equation}

In addition, the order relation of $\mathcal{F}(\mathcal{A})$ has an easy combinatorial description in terms of the sign vectors of the faces. Namely, $F_1\subseteq F_2$ in $\mathcal{F}(\mathcal{A})$ if and only if the sign vector of $F_1$ is obtained from that of $F_2$ by replacing some $+$ or $-$. Let $\le$ be the partial order on the set $\{+,-,0\}$ defined as $0<+,-$ with $+$ and $-$ incomparable. This further induces the {\em sign order} $\le$ on $\{\pm,0\}^m$, that is, for any members $\bm s=(s_1,\ldots,s_m),\bm w=(w_1,\ldots,w_m)\in\{\pm,0\}^m$,  $\bm s\le \bm w$ if  and only if $s_i\le w_i$ for each $i\in[m]$.
\begin{lemma}\label{Face-Sign-Order}
Let $\mathcal{A}=\{H_1,\ldots,H_m\}$ be a hyperplane arrangement in $\mathbb{R}^n$. Then $F_1\subseteq F_2$ in $\mathcal{F}(\mathcal{A})$ if and only if $\Sign_{\mathcal{A}}(F_1)\le\Sign_{\mathcal{A}}(F_2)$.
\end{lemma}
\begin{proof}
Suppose $\Sign_{\mathcal{A}}(F_1)\le\Sign_{\mathcal{A}}(F_2)$. This means
\begin{equation}\label{Relation}
I_{\mathcal{A}}(F_1)\supseteq I_{\mathcal{A}}(F_2),\quad J_{\mathcal{A}}(F_1)\subseteq J_{\mathcal{A}}(F_2)
\quad\And\quad K_{\mathcal{A}}(F_1)\subseteq K_{\mathcal{A}}(F_2).
\end{equation}
It follows from \eqref{Face-Form} that
\begin{equation*}
F_1=P\big(\bm a, I_{\mathcal{A}}(F_1), J_{\mathcal{A}}(F_1), K_{\mathcal{A}}(F_1)\big)\subseteq P\big(\bm a, I_{\mathcal{A}}(F_2), J_{\mathcal{A}}(F_2), K_{\mathcal{A}}(F_2)\big)=F_2.
\end{equation*}
Conversely, suppose $F_1,F_2$ are the faces of $\mathcal{A}$ and $F_1$ is a nonempty face of $F_2$. Then the active triple indices of $F_1$ and $F_2$ satisfy the same relations as in \eqref{Relation}. Following \eqref{Face-Sign-Vector}, we have $\Sign_{\mathcal{A}}(F_1)\le\Sign_{\mathcal{A}}(F_2)$, which completes the proof.
\end{proof}

With \autoref{Combinatorial-Face} and \autoref{Face-Sign-Order}, we further achieve the next theorem.
\begin{theorem}\label{Combinatorial-Equivalent0}
Let $\mathcal{A}=\{H_1,\ldots,H_m\}$ {\rm(}$\mathcal{A}'=\{H_1',\ldots,H_m'\}$, resp.{\rm)} be a hyperplane arrangement in $\mathbb{R}^n$ {\rm(}$\mathbb{R}^{n'}$, resp.{\rm)}. If $\mathcal{A}$ and $\mathcal{A}'$ are sign equivalent, then they are combinatorially and semi-lattice equivalent.
\end{theorem}
\begin{proof}
Recall from \eqref{Arrangement-Sign} that we have
\[
\Sign(\mathcal{A})=\big\{\Sign_{\mathcal{A}}(F)\mid F\in\mathcal{F}(\mathcal{A})\big\}\quad\And\quad \Sign(\mathcal{A}')=\big\{\Sign_{\mathcal{A}'}(F')\mid F'\in\mathcal{F}(\mathcal{A}')\big\}.
\]
Since $\Sign(\mathcal{A})=\Sign(\mathcal{A}')$, the above equations automatically give rise to a bijection $\Phi:\mathcal{F}(\mathcal{A})\to\mathcal{F}(\mathcal{A}')$ such that
\[
\Sign_{\mathcal{A}'}\big(\Phi(F)\big)=\Sign_{\mathcal{A}}(F),\quad\forall\, F\in\mathcal{F}(\mathcal{A}).
\]
So, to prove $\mathcal{F}(\mathcal{A})\cong\mathcal{F}(\mathcal{A}')$, it can be reduced to showing that $F_1\subseteq F_2$ in $\mathcal{F}(\mathcal{A})$ if and only if $\Sign_{\mathcal{A}}(F_1)\le\Sign_{\mathcal{A}}(F_2)$. It follows from \autoref{Face-Sign-Order} that this holds. Hence, we arrive at $\mathcal{F}(\mathcal{A})\cong\mathcal{F}(\mathcal{A}')$. Combining \autoref{Combinatorial-Face}, we conclude $L(\mathcal{A})\cong L(\mathcal{A}')$.
\end{proof}

Immediately, applying \autoref{Combinatorial-Equivalent0} to (i), (iii) and (iv) in \autoref{Three-Type-Sign}, we deduce that (ii) and (iii) hold in \autoref{Three-Type-Normal-Combinatorial}. Additionally, we quickly derive the following result from \autoref{Combinatorial-Face} and \autoref{Equivalence-Arrangement}.
\begin{corollary}\label{Combinatorial-Equivalent}
Let $\mathcal{A}$ and $\mathcal{A}'$ be  hyperplane arrangements in $\mathbb{R}^n$. If $\mathcal{A}$ and $\mathcal{A}'$ are normally equivalent, then they are combinatorially and semi-lattice equivalent.
\end{corollary}

We conclude this section by summarizing the following relations:
\begin{equation}\label{Four-Relations}
\mathcal{A}, \mathcal{A}' \mbox{ are normally or sign equivalent}\Rightarrow \mathcal{F}(\mathcal{A})\cong\mathcal{F}(\mathcal{A}')\Rightarrow
L(\mathcal{A})\cong L(\mathcal{A}').
\end{equation}
\subsection{Parallel translations}\label{Sec-4-2}
Based on the study in \autoref{Sec-3} and \autoref{Sec-4-1}, we further examine the normal, combinatorial and sign equivalence relations on parallel translations. We begin by proving the result (i) in \autoref{Three-Type-Normal-Combinatorial} that for each face $F$ of $\delta\mathcal{A}_{\bm o}$, the parallel translations $\mathcal{A}_{\bm a}$ and $\mathcal{A}_{\bm b}$ are normally and combinatorially equivalent for any $\bm a,\bm b\in\relint(F)$.
\begin{proof}[Proof of (i) in \autoref{Three-Type-Normal-Combinatorial}]
By \autoref{Combinatorial-Equivalent}, it reduces to showing that $\mathcal{A}_{\bm a}$ and $\mathcal{A}_{\bm b}$ are normally equivalent. Note that the expressing form of faces in \eqref{Face-Form} induces a map
\begin{equation}\label{EQ0}
\Psi:\mathcal{F}(\mathcal{A}_{\bm a})\to\mathcal{F}(\mathcal{A}_{\bm b}),\quad G\mapsto
\Psi(G)=P\big(\bm b,I_{\mathcal{A}_{\bm a}}(G),J_{\mathcal{A}_{\bm a}}(G),K_{\mathcal{A}_{\bm a}}(G)\big).
\end{equation}
Since $G=P\big(\bm a,I_{\mathcal{A}_{\bm a}}(G),J_{\mathcal{A}_{\bm a}}(G),K_{\mathcal{A}_{\bm a}}(G)\big)$, we have $\Phi(G)\ne\emptyset$ by  \autoref{Nonempty}. This implies that $\Phi(G)$ is a face of $\mathcal{A}_{\bm b}$, that is, $\Psi$ is well-defined. It follows from $I_{\mathcal{A}_{\bm b}}\big(\Psi(G)\big)=I_{\mathcal{A}_{\bm a}}(G)$ in \autoref{Face-Bijection} that we arrive at
\begin{equation}\label{Index-Equation}
\big(I_{\mathcal{A}_{\bm b}}(\Psi(G)),J_{\mathcal{A}_{\bm b}}(\Psi(G)),K_{\mathcal{A}_{\bm b}}(\Psi(G))\big)=\big(I_{\mathcal{A}_{\bm a}}(G),J_{\mathcal{A}_{\bm a}}(G),K_{\mathcal{A}_{\bm a}}(G)\big).
\end{equation}
Then $G$ and $\Psi(G)$ are normally equivalent via \autoref{Normally-Combinatorially}.

Next we show that $\Psi$ is bijective. According to \eqref{Face-Form}, we deduce that two faces $F_1$ and $F_2$ of $\mathcal{A}_{\bm a}$ are identical if and only if $F_1$ and $F_2$ share the same active triple index. Then, when $F_1\ne F_2$, they have different  active triple indices. From \eqref{Index-Equation}, we get $\Psi(F_1)$ and $\Psi(F_2)$ have distinct active triple indices as well. This implies $\Psi(F_1)\ne\Psi(F_2)$. Namely, $\Psi$ is injective. Similarly, define a map
\[
\Psi^{-1}:\mathcal{F}(\mathcal{A}_{\bm b})\to\mathcal{F}(\mathcal{A}_{\bm a}),\quad G\mapsto\Psi^{-1}(G)=P\big(\bm a,I_{\mathcal{A}_{\bm b}}(G),J_{\mathcal{A}_{\bm b}}(G),K_{\mathcal{A}_{\bm b}}(G)\big).
\]
As the roles of $\mathcal{A}_{\bm a}$ and $\mathcal{A}_{\bm b}$ can be interchanged, we can also ascertain that $\Psi^{-1}$ is an injection. Furthermore, the equation \eqref{Index-Equation} guarantees that $\Psi\big(\Psi^{-1}(G)\big)=G$ for any face $G$ of $\mathcal{A}_{\bm b}$. Hence, $\Psi$ is a surjection. Up to now we have proved that $\Psi$ is a bijection.

It remains to prove that $\Psi$ is an order-preserving map. Suppose $G_0$ is a proper face of $G$. Then $G_0$ is also a face of $\mathcal{A}_{\bm a}$.
From \eqref{Index-Equation}, $G_0$ and $\Psi(G_0)$ ($G$ and $\Psi(G)$, resp.) have the same active triple index. It follows from \autoref{Face-Sign-Order} that $\Psi(G_0)$ is a face of $\Psi(G)$, i.e., $\Psi$ is order-preserving. We complete the proof.
\end{proof}

To verify that the results (i) and (ii) in \autoref{Three-Type-Sign} are equivalent mutually, we build the connection between the circuit vectors and  sign vectors of  parallel translations of $\mathcal{A}_{\bm o}$. Given a real number $x$, without causing confusion, we also use the notation $\Sign( x)$ to denote the sign of $x$, that is, $\Sign(x)$ equals $+$ if $x>0$, $-$ if $x<0$, $0$ else. Associated with a circuit vector $\bm c^C=(c_1,\ldots,c_m)$ of $C\in\mathcal{C}(\mathcal{A}_{\bm o})$,  we denote by
\[
S_{\bm c^C}:=\big\{\bm s=(s_1,\ldots,s_m)\in\{\pm,0\}^m\mid s_i=\Sign(c_i),\forall\, i\in C\big\}.
\]
For convenience, let
\[
C^+:=\{i\in C\mid c_i>0\}\quad\And\quad C^-:=\{i\in C\mid c_i<0\}.
\]
\begin{lemma}\label{Separate-Cells}
Let $\bm c^C=(c_1,\ldots,c_m)$ be a circuit vector of $C\in\mathcal{C}(\mathcal{A}_{\bm o})$, $\bm a\in \mathring{H}_{\bm c^C}^-$ and $\bm b\in H_{\bm c^C}^+$. Then
\[
\Sign(\mathcal{A}_{\bm a}) \cap S_{\bm c^C}\ne\emptyset\quad \And \quad\Sign(\mathcal{A}_{\bm b}) \cap S_{\bm c^C}=\emptyset.
\]
\end{lemma}
\begin{proof}
Suppose $\bm y\in\mathbb{R}^{|C|}$ satisfies $\bm yU_C=0$. Since $\{\bm u_i\mid i\in C\}$ is the minimum linearly dependent set via \eqref{circuit}, $\bm y=\lambda \bm c^C_C$ for some non-zero real number $\lambda$ . Recall from \eqref{circuit-vector} that $C=C^+\sqcup C^-$. It follows from
\autoref{General-Farkas-lemma} that $\mathring{P}(\bm a_C,\emptyset, C^-,C^+)\ne\emptyset$. This implies $\Sign(\mathcal{A}_{\bm a}) \cap S_{\bm c^C}\ne\emptyset$. Likewise, we can obtain $\Sign(\mathcal{A}_{\bm b}) \cap S_{\bm c^C}=\emptyset$.
\end{proof}

We now proceed to show that the parallel translations $\mathcal{A}_{\bm a}$ and $\mathcal{A}_{\bm b}$ have the same sign if and only if the vectors $\bm a,\bm b$ come from the same open face of $\delta\mathcal{A}_{\bm o}$ in \autoref{Three-Type-Sign}.
\begin{proof}[ Proof of (i) $\Leftrightarrow$ (ii) in \autoref{Three-Type-Sign}]
To obtain (i) $\Rightarrow$ (ii), let the map $\Psi:\mathcal{F}(\mathcal{A}_{\bm a})\to\mathcal{F}(\mathcal{A}_{\bm b})$ be defined as in \eqref{EQ0}.  Since $\bm a,\bm b\in\relint(F)$ for $F\in\mathcal{F}(\delta\mathcal{A}_{\bm o})$, $\Psi$ is a bijection by the proof of (i) in \autoref{Three-Type-Normal-Combinatorial}. Combining \eqref{Face-Sign-Vector} and \eqref{Index-Equation}, we have $\Sign(\mathcal{A}_{\bm a})=\Sign(\mathcal{A}_{\bm b})$ via \eqref{Arrangement-Sign}.

Prove (ii) $\Rightarrow$ (i) by negation. Suppose $\bm a\in \mathring{H}_{\bm c^C}^-$ and $\bm b\in H_{\bm c^C}^+$  for some $C\in\mathcal{C}(\mathcal{A}_{\bm o})$ and $H_{\bm c^C}\in\delta\mathcal{A}_{\bm o}$. It follows from \autoref{Separate-Cells} that $\Sign(\mathcal{A}_{\bm a})\ne\Sign(\mathcal{A}_{\bm b})$, which contradicts the assumption $\Sign(\mathcal{A}_{\bm a})=\Sign(\mathcal{A}_{\bm b})$. We finish the proof.
\end{proof}

\autoref{Geometric-Face} further illustrates that the normal and sign equivalences of parallel translations of $\mathcal{A}_{\bm o}$ are consistent whenever $\mathcal{A}_{\bm o}$ is not a multi-arrangement. The following results are necessary to establish this.
\begin{lemma}\label{Separate-Regions}
Let $\bm c^C=(c_1,\ldots,c_m)$  be a circuit vector of $C\in\mathcal{C}(\mathcal{A}_{\bm o})$ and $\bm a\in \mathring{H}_{\bm c^C}^-$. Then
$P(\bm a_C,\emptyset, C^-,C^+)$ is a closed region of $\mathcal{A}_{\bm a_C}:=\{H_{\bm u_j,a_j}\mid j\in C\}$ and $\dim\big(P(\bm a_C,\emptyset, C^-,C^+)\cap H_{\bm u_i,a_i}\big)=n-1$ for each $i\in C$.
\end{lemma}
\begin{proof}
From the proof of \autoref{Separate-Cells}, we have $\mathring{P}(\bm a_C,\emptyset, C^-,C^+)\ne\emptyset$. This means that $P(\bm a_C,\emptyset, C^-,C^+)$ is a closed region of $\mathcal{A}_{\bm a_C}$.  For simplicity, let $R:=P(\bm a_C,\emptyset, C^-,C^+)$. Given a member $i\in C$, suppose $R\subseteq H_{\bm u_i,a_i}^-$. We prove  $\dim(R\cap H_{\bm u_i,a_i})=n-1$ by negation. Suppose $\dim(R\cap H_{\bm u_i,a_i})\le n-2$. If $\dim(R\cap H_{\bm u_i,a_i})=-1$, i.e., $R\cap H_{\bm u_i,a_i}=\emptyset$, then $R\subseteq\mathring{H}_{\bm u_i,a_i}^-$. Equivalently, the following inequality systems
\[
\langle\bm u_j,\bm x\rangle<a_j \For j\in C^+,\quad\langle\bm u_j,\bm x\rangle>a_j \For j\in C^-\setminus\{i\}\quad\And\quad \langle\bm u_i,\bm x\rangle<a_i,
\]
have a solution. As with the proof of \autoref{Separate-Cells}, this is impossible. So $\dim( R\cap H_{\bm u_i,a_i})\ge 0$. Let $G= R\cap H_{\bm u_i,a_i}$ and $\dim(G)=k\le n-2$. From \autoref{Facet-Characterization} and the property (ii) in \autoref{Face-Face}, we have $G=\bigcap_{j\in I_G}(H_{\bm u_j,a_j}\cap R)$ and $i\notin I_G$, where
\[
I_G=\big\{j\in C\mid G\subseteq H_{\bm u_j,a_j}\cap R\And\dim(H_{\bm u_j,a_j}\cap R)=n-1\big\}.
\]
Since $\big\{\bm u_j\mid j\in C\setminus\{i\}\big\}$ is a linearly independent set, then
\[
2\le|I_G|=n-k\le|C|-1\le n.
\]
If $n-k\le |C|-2$, then $I_G\sqcup\{i\}$ contains a circuit $C_0$ of $\mathcal{A}_{\bm o}$ since $G\subseteq H_{\bm u_i,a_i}$. Then $C_0$ is a nonempty proper subset of $C$, a contradiction. If $n-k=|C|-1$, then $\bigcap_{j\in C}H_{\bm u_j,a_j}=G\ne\emptyset$. This means $\bm a\in H_{\bm c^C}$ contradicting   $\bm a\in\mathring{H}_{\bm c^C}^-$. So we obtain $\dim(R\cap H_{\bm u_i,a_i})=n-1$, which finishes the proof.
\end{proof}

\begin{lemma}\label{Restriction-Geometric-Iquivalence}
Suppose $\mathcal{A}_{\bm o}$ is not a multi-arrangement. Let $\bm a,\bm b\in\mathbb{R}^m$ and the parallel translations $\mathcal{A}_{\bm a}$ and $\mathcal{A}_{\bm b}$ of $\mathcal{A}_{\bm o}$ be normally equivalent with a normal map $\Psi$. Then each map $\Psi_i$ defined as
\[
\Psi_i:\mathcal{F}(\mathcal{A}_{\bm a}/H_{\bm u_i,a_i})\to\mathcal{F}(\mathcal{A}_{\bm b}/H_{\bm u_i,b_i}),\quad F\mapsto\Psi_i( F)=\Psi( F),
\]
is a normal map.
\end{lemma}
\begin{proof}
Recall from \autoref{Normal-Cone1} that the normal cones of all closed regions of $\mathcal{A}_{\bm a}/H_{\bm u_i,a_i}$ and $\mathcal{A}_{\bm b}/H_{\bm u_i,b_i}$ are exactly the same vector space $\Span(\bm u_i)$. Together with that $\mathcal{A}_{\bm o}$ is not a multi-arrangement, then the map $\Psi$ bijectively sends the closed regions of $\mathcal{A}_{\bm a}/H_{\bm u_i,a_i}$ to the closed regions of $\mathcal{A}_{\bm b}/H_{\bm u_i,b_i}$. Since every face of $\mathcal{A}_{\bm a}/H_{\bm u_i,a_i}$ is also a face of some closed region of $\mathcal{A}_{\bm a}/H_{\bm u_i,a_i}$, the order-preserving property of $\Psi$ guarantees that $\Psi$ sends the faces of $\mathcal{A}_{\bm a}/H_{\bm u_i,a_i}$ to the faces of $\mathcal{A}_{\bm b}/H_{\bm u_i,b_i}$. Namely, $\Psi_i$ is well-defined. According to the definition of $\Psi_i$ and the property of $\Psi$, we can easily check that $\Psi_i$ is a normal map.
\end{proof}

It is worth noting that given a hyperplane arrangement $\mathcal{A}=\{H_1,\ldots,H_m\}$ in $\mathbb{R}^n$, a face $ F$ of $\mathcal{A}$ can be also represented as
\begin{equation}\label{Separate-Cells1}
 F=P\big(\bm a_{\mathcal{I}(\mathcal{A}, F)},I_{\mathcal{A}}( F),\bar{J}_{\mathcal{A}}( F),\bar{K}_{\mathcal{A}}( F)\big),
\end{equation}
where $\bar{J}_{\mathcal{A}}( F):=\big\{j\in J_{\mathcal{A}}(F)\mid  F\cap H_{\bm u_j,a_j}\ne\emptyset\big\}$, $\bar{K}_{\mathcal{A}}( F):=\big\{k\in K_{\mathcal{A}}(F)\mid F\cap H_{\bm u_k,a_k}\ne\emptyset\big\}$, and
\[
\mathcal{I}(\mathcal{A}, F)=I_{\mathcal{A}}( F)\sqcup \bar{J}_{\mathcal{A}}( F)\sqcup \bar{K}_{\mathcal{A}}( F).
\]
The triple $\big(I_{\mathcal{A}}(F),\bar{J}_{\mathcal{A}}(F),\bar{K}_{\mathcal{A}}(F)\big)$ is referred to as a {\em valid active triple index} of $F$ induced by $\mathcal{A}$. Following \autoref{Normal-Cone1}, it is a obvious fact that the normal cone of  every nonempty face $F_0$ of $F$ is completely determined by the valid active triple index. Thus, the following result is observed.
\begin{lemma}\label{Geometric-Iquivalence-Index}
Suppose $\mathcal{A}_{\bm o}$ is not a multi-arrangement. Let $\bm a,\bm b\in\mathbb{R}^m$. Then the parallel translations $\mathcal{A}_{\bm a}$ and $\mathcal{A}_{\bm b}$ of $\mathcal{A}_{\bm o}$ are normally equivalent if and only if there is an order-preserving bijection $\Psi:\mathcal{F}(\mathcal{A}_{\bm a})\to\mathcal{F}(\mathcal{A}_{\bm b})$ such that  $F,\Psi(F)$ have the same valid active triple index for any $F\in\mathcal{F}(\mathcal{A}_{\bm a})$.
\end{lemma}
\begin{proof}
The sufficiency is trivial. For necessity, let $\Psi$ be a normal map from $\mathcal{A}_{\bm a}$ to $\mathcal{A}_{\bm b}$. According to \autoref{Equivalence-Arrangement}, it remains to show  that $F$ and $\Psi(F)$ have the same valid active triple index for any $F\in\mathcal{F}(\mathcal{A}_{\bm a})$. First $i\in I_{\mathcal{A}_{\bm a}}( F)$ implies $ F\in \mathcal{F}(\mathcal{A}_{\bm a}/H_{\bm u_i,a_i})$. From \autoref{Restriction-Geometric-Iquivalence}, $\Psi( F)\in \mathcal{F}(\mathcal{A}_{\bm b}/H_{\bm u_i,b_i})$, i.e., $i\in I_{\mathcal{A}_{\bm b}}(\Psi( F))$. Hence, $I_{\mathcal{A}_{\bm a}}( F)\subseteq I_{\mathcal{A}_{\bm b}}(\Psi( F))$. By symmetry, we have $I_{\mathcal{A}_{\bm a}}( F)\supseteq I_{\mathcal{A}_{\bm b}}(\Psi( F))$.  So
\[
I_{\mathcal{A}_{\bm a}}( F)=I_{\mathcal{A}_{\bm b}}(\Psi( F)).
\]
For any $j\in\bar{J}_{\mathcal{A}_{\bm a}}( F)$, $F( F,\bm u_j)=H_{\bm u_j,a_j}\cap F$ is a proper face of $ F$ and $ F\subseteq H_{\bm u_j,a_j}^-$. Since the map $\Psi$ is order-preserving and $F,\Psi(F)$ are normally equivalent, we have
\[
\Psi\big(F( F,\bm u_j)\big)=F\big(\Psi( F),\bm u_j\big)=H_{\bm u_j,h_{\Psi( F)}(\bm u_j)}\cap\Psi( F)
\]
is a proper face of $\Psi( F)$ and $\Psi( F)\subseteq H_{\bm u_j,h_{\Psi( F)}(\bm u_j)}^-$. Applying \autoref{Restriction-Geometric-Iquivalence} again, we get $H_{\bm u_j,h_{\Psi( F)}(\bm u_j)}\cap\Psi( F)\subseteq H_{\bm u_j,b_j}\cap\Psi( F)$. This means  $H_{\bm u_j,b_j}=H_{\bm u_j,h_{\Psi( F)}(\bm u_j)}$, i.e., $j\in \bar{J}_{\mathcal{A}_{\bm b}}(\Psi( F))$. So $\bar{J}_{\mathcal{A}_{\bm a}}( F)\subseteq \bar{J}_{\mathcal{A}_{\bm b}}(\Psi( F))$. By symmetry, we can arrive at $\bar{J}_{\mathcal{A}_{\bm a}}( F)\supseteq \bar{J}_{\mathcal{A}_{\bm b}}(\Psi( F))$. Therefore,
\[
\bar{J}_{\mathcal{A}_{\bm a}}( F)=\bar{J}_{\mathcal{A}_{\bm b}}(\Psi( F)).
\]
Likewise, we can also acquire $\bar{K}_{\mathcal{A}_{\bm a}}( F)=\bar{K}_{\mathcal{A}_{\bm b}}(\Psi( F))$. We complete the proof.
\end{proof}

Given a hyperplane arrangement $\mathcal{A}=\{H_1,\ldots,H_m\}$ in $\mathbb{R}^n$, let $\mathcal{I}(\mathcal{A})$ be the collection of the valid active triple indices of all faces of $\mathcal{A}$. Immediately, \eqref{Separate-Cells1} establishes a natural one-to-one correspondence between the face poset $\mathcal{F}(\mathcal{A})$ and $\mathcal{I}(\mathcal{A})$ by mapping each face of $\mathcal{A}$ to its corresponding valid active triple index. This leads to the following lemma.
\begin{lemma}\label{Deletion-Geometric-Iquivalence}
Suppose $\mathcal{A}_{\bm o}$ is not a multi-arrangement. Let $\bm a,\bm b\in\mathbb{R}^m$ and the parallel translations $\mathcal{A}_{\bm a}$ and $\mathcal{A}_{\bm b}$ of $\mathcal{A}_{\bm o}$ be normally equivalent. Then $\mathcal{A}_{\bm a}\setminus H_{\bm u_i,a_i}$ and $\mathcal{A}_{\bm b}\setminus H_{\bm u_i,b_i}$ are normally equivalent for each $i\in[m]$.
\end{lemma}
\begin{proof}
Let $\Psi$ be a normal map from $\mathcal{A}_{\bm a}$ to $\mathcal{A}_{\bm b}$. For every $i\in[m]$, define a map $\Psi^i: \mathcal{F}(\mathcal{A}_{\bm a}\setminus H_{\bm u_i,a_i})\to \mathcal{F}(\mathcal{A}_{\bm b}\setminus H_{\bm u_i,b_i})$  such that
\[
F\mapsto \Psi^i(F)=P\big(\bm b_{\mathcal{I}(\mathcal{A}_{\bm a}\setminus H_{\bm u_i,a_i}, F)},I_{\mathcal{A}_{\bm a}\setminus H_{\bm u_i,a_i}}( F),\bar{J}_{\mathcal{A}_{\bm a}\setminus H_{\bm u_i,a_i}}( F),\bar{K}_{\mathcal{A}_{\bm a}\setminus H_{\bm u_i,a_i}}( F)\big).
\]
The proof now turns to showing that $\Psi^i$ is a normal map. We first prove that $\Psi^i$ is well-defined. Note that $\mathcal{I}(\mathcal{A}_{\bm a}\setminus H_{\bm u_i,a_i})$ ($\mathcal{I}(\mathcal{A}_{\bm b}\setminus H_{\bm u_i,b_i})$, resp.) is obtained from $\mathcal{I}(\mathcal{A}_{\bm a})$ ($\mathcal{I}(\mathcal{A}_{\bm b})$, resp.) by removing the index $i$ from the valid active triple index of every face of $\mathcal{A}_{\bm a}$ ($\mathcal{A}_{\bm b}$, resp.). Following  \autoref{Geometric-Iquivalence-Index},  we have $\mathcal{I}(\mathcal{A}_{\bm a})=\mathcal{I}(\mathcal{A}_{\bm b})$. Then we further deduce
\begin{equation}\label{EQ00}
\mathcal{I}(\mathcal{A}_{\bm a}\setminus H_{\bm u_i,a_i})=\mathcal{I}(\mathcal{A}_{\bm b}\setminus H_{\bm u_i,b_i}).
\end{equation}
This means that $\Psi^i(F)$ is indeed a face of $\mathcal{A}_{\bm b}\setminus H_{\bm u_i,b_i}$ for each face of $\mathcal{A}_{\bm a}\setminus H_{\bm u_i,a_i}$. Namely, $\Psi^i$ is well-defined. Furthermore, we derive from \eqref{Separate-Cells1} and \eqref{EQ00} that $\Psi^{i}$ is a bijection. Suppose $F_1,F_2$ are the faces of $\mathcal{A}_{\bm a}$ and $F_1$ is a nonempty face of $F_2$. Then
\[
I_{\mathcal{A}_{\bm a}}(F_1)\supseteq I_{\mathcal{A}_{\bm a}}(F_2),\quad \bar{J}_{\mathcal{A}_{\bm a}}(F_1)\subseteq \bar{J}_{\mathcal{A}_{\bm a}}(F_2)\quad\And\quad \bar{K}_{\mathcal{A}_{\bm a}}(F_1)\subseteq \bar{K}_{\mathcal{A}_{\bm a}}(F_2).
\]
Together with the definition of $\Psi^i$ and \eqref{EQ00}, we obtain
\[
I_{\mathcal{A}_{\bm b}}(\Psi^i(F_1))\supseteq I_{\mathcal{A}_{\bm b}}(\Psi^i(F_2)),\; {\bar J}_{\mathcal{A}_{\bm b}}(\Psi^i(F_1))\subseteq \bar{J}_{\mathcal{A}_{\bm b}}(\Psi^i(F_2))\And \bar{K}_{\mathcal{A}_{\bm b}}(\Psi^i(F_1))\subseteq \bar{K}_{\mathcal{A}_{\bm b}}(\Psi^i(F_2)).
\]
This implies that $\Psi^i(F_1)$ is a face of $\Psi^i(F_2)$ by \eqref{Separate-Cells1}. Namely, $\Psi^i$ is order-preserving. It follows from \autoref{Geometric-Iquivalence-Index} that $\Psi^i$ is a normal map. We finish the proof.
\end{proof}

We now return to \autoref{Geometric-Face}.
\begin{corollary}\label{Geometric-Face}
Suppose $\mathcal{A}_{\bm o}$ is not a multi-arrangement. Let $\bm a,\bm b\in\mathbb{R}^m$. Then $\mathcal{A}_{\bm a}$ and $\mathcal{A}_{\bm b}$ are normally equivalent if and only if $\mathcal{A}_{\bm a}$ and $\mathcal{A}_{\bm b}$ are sign equivalent.
\end{corollary}
\begin{proof}
First (i) in \autoref{Three-Type-Normal-Combinatorial} and (i), (ii) in \autoref{Three-Type-Sign} directly obtain the sufficiency. For necessity, applying (i) and (ii) in \autoref{Three-Type-Sign} again, it remains to show that the vectors $\bm a,\bm b$ belong to the same open face of $\delta\mathcal{A}_{\bm o}$. We prove it by contradiction. Suppose $\bm a\in\mathring{H}_{\bm c^C}^-$ and $\bm b\in H_{\bm c^C}^+$ for some $C\in\mathcal{C}(\mathcal{A}_{\bm o})$ and $H_{\bm c^C}\in\delta\mathcal{A}_{\bm o}$. It follows from \autoref{Separate-Regions} that $P(\bm a_C,\emptyset, C^-,C^+)$ is a closed region of the subarrangement $\mathcal{A}_{\bm a_C}$ and $\dim\big(H_{\bm u_i,a_i}\cap P(\bm a_C,\emptyset, C^-,C^+)\big)=n-1$ for each $i\in C$. Note from \autoref{Deletion-Geometric-Iquivalence} that  subarrangements $\mathcal{A}_{\bm a_C}$ and $\mathcal{A}_{\bm b_C}$ are normally equivalent. This implies that $P(\bm b_C,\emptyset, C^-,C^+)$ must be a closed region of the subarrangement $\mathcal{A}_{\bm b_C}$ and $\dim\big(H_{\bm u_i,b_i}\cap P(\bm b_C,\emptyset, C^-,C^+)\big)=n-1$ for each $i\in C$. This further means $\Sign(\mathcal{A}_{\bm b})\cap S_{\bm c^C}\ne\emptyset$, which contradicts \autoref{Separate-Cells}. So we complete the proof.
\end{proof}

In general, the normally equivalent real hyperplane arrangements are not necessarily sign equivalent. For example, let
\[
\mathcal{A}=\{H_1:x=0,H_2:y=0,H_3:x+y=2\}
\]
and
\[
\mathcal{A}'=\{H_1':y=0,H_2':x=0,H_3':-x-y=-2\}
\]
be hyperplane arrangements in $\mathbb{R}^2$. Clearly $\mathcal{A}$ and $\mathcal{A}'$ are normally equivalent, but their signs are not same. However, $\mathcal{A}$ and $\mathcal{A}'$ share the same sign up to reorientation and relabeling. For general result, see \autoref{Geometric-Sign}. Given a hyperplane $H_{\bm u,a}$ in $\mathbb{R}^n$, denoted by $-H_{\bm u,a}:=H_{-\bm u,-a}$. Let $\mathcal{A}=\{H_1,\ldots,H_m\}$ be a hyperplane arrangement in $\mathbb{R}^n$. For any subset $A\subseteq[m]$, we use $_{-A}\mathcal{A}$ to denote a {\em reorientation} of $\mathcal{A}$ obtained from $\mathcal{A}$ by reversing the orientations on $A$, that is,
\[
_{-A}\mathcal{A}:=\mathcal{A}\setminus\{H_i\mid i\in A\}\cup\{-H_i\mid i\in A\}.
\]
Given a permutation $\pi\in\mathfrak{S}_m$, define $\pi(\mathcal{A}):=\big\{H_{\pi(i)}\mid i\in[m]\big\}$, which is referred to as a {\em relabeling} of $\mathcal{A}$.
\begin{theorem}\label{Geometric-Sign}
Let $\mathcal{A}=\big\{H_{\bm u_i,a_i}\mid i\in[m]\big\}$ and $\mathcal{A}'=\big\{H_{\bm v_i,b_i}\mid i\in[m]\big\}$ be two hyperplane arrangements in $\mathbb{R}^n$. Then $\mathcal{A}$ and $\mathcal{A}'$ are normally equivalent if and only if $\mathcal{A}$ and $\mathcal{A}'$ are the parallel translations of $\mathcal{A}_{\bm o}$ with the same sign up to reorientation and relabeling.
\end{theorem}
\begin{proof}
For sufficiency, without loss of generality, suppose the relabeling $\pi$ is the identity element of $\mathfrak{S}_m$ and the reorientation $A=\emptyset$. Together with (i) in \autoref{Three-Type-Normal-Combinatorial} and (i), (ii) in \autoref{Three-Type-Sign}, we have that $\mathcal{A}$ and $\mathcal{A}'$ are normally equivalent.

For necessity, it reduces to proving that there exist a subset $A\subseteq[m]$ and a permutation $\pi\in\mathfrak{S}_m$ such that $\mathcal{A}$, $_{-A}\pi(\mathcal{A}')$ are the parallel translations of $\mathcal{A}_{\bm o}$ and $\Sign(\mathcal{A})=\Sign\big(_{-A}\pi(\mathcal{A}')\big)$.
Since $\mathcal{A}$ and $\mathcal{A}'$ are normally equivalent with a normal map $\Psi$, we can derive from \autoref{Normal-Cone1} and \autoref{Restriction-Geometric-Iquivalence} that there exist a permutation $\pi\in\mathfrak{S}_m$ and $m$ nonzero real numbers $\lambda_1,\ldots,\lambda_m$ such that
\begin{itemize}
\item [{\rm (a)}] $\mathcal{A}$ and $\mathcal{A}_{\bm d}$ are normally equivalent with the same normal map $\Psi$;
\item [{\rm (b)}] $\mathcal{A}/H_{\bm u_i,a_i}$ and $\mathcal{A}_{\bm d}/H_{\bm u_i,d_i}$ are normally equivalent with the normal map $\Psi _i$ defined as in \autoref{Restriction-Geometric-Iquivalence},
\end{itemize}
where  $\bm u_i=\lambda_i\bm v_{\pi^{-1}(i)}$ for each $i\in[m]$ and $d_i=\lambda_ib_{\pi^{-1}(i)}$ for each $i\in[m]$. Let $A=\big\{i\in[m]\mid \lambda_i<0\big\}$. Then $_{-A}\pi(\mathcal{A}')$ can be rewritten as $_{-A}\pi(\mathcal{A}')=\big\{H_{\bm u_i,d_i}\mid i\in[m]\big\}$. Hence, $\mathcal{A}=\mathcal{A}_{\bm a}$ and  $_{-A}\pi(\mathcal{A}')=\mathcal{A}_{\bm d}$ are the parallel translations of $\mathcal{A}_{\bm o}$. In this sense, we assert
\[
\Sign(\mathcal{A}_{\bm a})=\Sign(\mathcal{A}_{\bm d}).
\]
Following (i) and (ii) in \autoref{Three-Type-Sign}, it is equivalent to showing that $\bm a$ and $\bm d$ belong to the same open face of $\delta\mathcal{A}_{\bm o}$. Otherwise, suppose $\bm a\in\mathring{H}_{\bm c^C}^-$ and $\bm d\in H_{\bm c^C}^+$ for some $C\in\mathcal{C}(\mathcal{A}_{\bm o})$. Then $|C|\ge2$. If $|C|\ge 3$, then $\Sign(\mathcal{A}_{\bm a})=\Sign(\mathcal{A}_{\bm d})$ can be proved by the same method as employed in \autoref{Geometric-Face}. If $|C|=2$, suppose $C=\{1,2\}$, $\bm u_1=\bm u_2$ and $\bm c^C=(c_1,\ldots,c_m)$, where each $c_i$ equals $0$ except for $c_1=1$ and $c_2=-1$. From \autoref{Separate-Cells}, $\mathcal{A}_{\bm a_C}$ and $\mathcal{A}_{\bm d_C}$  can be illustrated as follows:
\begin{figure}[H]
\centering
\begin{tikzpicture}[scale=1,line width=1pt]
\draw [black,-](0,0) -- (0,1) node [above, black] {$H_{\bm u_1,a_1}$};
\draw [black,-](1.5,0) -- (1.5,1) node [above, black] {$H_{\bm u_2,a_2}$};
\draw (0.75,-0.5) node[black] {$\mathcal{A}_{\bm a_C}$};
\draw [black,densely dotted](3,-0.8) -- (3,1.7);
\draw [black,-](5.5,0) -- (5.5,1) node [above, black] {$H_{\bm u_1,d_1}=H_{\bm u_2,d_2}$};
\draw (5.5,-0.5) node[black] {$\mathcal{A}_{\bm d_C}\,(\bm d\in H_{\bm c^C})$};
\draw (7.25,0.5) node[black]{or};
\draw [black,-](8.5,0) -- (8.5,1) node [above, black] {$H_{\bm u_2,d_2}$};
\draw [black,-](10,0) -- (10,1) node [above, black] {$H_{\bm u_1,d_1}$};
\draw (9.5,-0.5) node[black] {$\mathcal{A}_{\bm d_C}\,(\bm d\in \mathring{H}^+_{\bm c^C})$};
\draw (5,-1.5) node[black] {Fig-1};
\end{tikzpicture}
\end{figure}
\noindent Similar to \autoref{Deletion-Geometric-Iquivalence}, we can deduce that $\mathcal{A}_{\bm a_C}$ and $\mathcal{A}_{\bm d_C}$ are normally equivalent. By Fig-1, the case $\bm d\in H_{\bm c^C}$ is impossible. When $\bm d\in \mathring{H}^+_{\bm c^C}$,  from (a) and (b), we achieve
\[
\Psi\big(F( R,\bm u_2)\big)=F\big(\Psi( R),\bm u_2\big)\subseteq H_{\bm u_2,d_2}
\]
for any closed region $R\subseteq H_{\bm u_2,a_2}^-$ of $\mathcal{A}_{\bm a}$ bounded by the hyperplane $H_{\bm u_2,a_2}$. It implies that  $\mathcal{A}_{\bm d_C}$ must become the following form:
\begin{figure}[H]
\centering
\begin{tikzpicture}[scale=1,line width=1pt]
\draw [black,-](0,0) -- (0,1) node [above, black] {$H_{\bm u_1,d_1}$};
\draw [black,-](1.5,0) -- (1.5,1) node [above, black] {$H_{\bm u_2,d_2}$};
\draw (0.75,-0.5) node[black] {$\mathcal{A}_{\bm d_C}$};
\end{tikzpicture}
\end{figure}
\noindent This is in contradiction with Fig-1. So,  the vectors $\bm a$ and $\bm d$ belong to the same open face of $\delta\mathcal{A}_{\bm o}$ in this case. We complete the proof.
\end{proof}
\subsection{Conings and elementary lifts}\label{Sec-4-3}
We begin by proving the remaining results in \autoref{Three-Type-Sign}. Let $\mathcal{A}=\{H_1,\ldots,H_m\}$ be an affine arrangement in $\mathbb{R}^n$. In fact, the faces of $\mathcal{A}$ are in correspondence with those faces of $c\mathcal{A}$ that are not contained in the negative half-space $K_0^-$. More specifically, let
\[
\mathcal{F}^+(c\mathcal{A}):=\big\{F\in\mathcal{F}(c\mathcal{A})\mid F\subseteq K_0^+\And F\nsubseteq K_0\big\}
\]
and
\[
\Sign^+(c\mathcal{A}):=\big\{\Sign_{c\mathcal{A}}(F)\mid F\in\mathcal{F}^+(c\mathcal{A})\big\}.
\]
Then there is a one-to-one correspondence $\kappa$ between $\mathcal{F}(\mathcal{A})$ to $\mathcal{F}^+(c\mathcal{A})$ sending each face $F$ of $\mathcal{A}$ to the face $\kappa(F)$ of $c\mathcal{A}$ given by
\[
\kappa(F):=\Big(\bigcap_{i\in I_{\mathcal{A}}(F)}cH_{\bm u_i,a_i}\Big)\bigcap\Big(\bigcap_{j\in J_{\mathcal{A}}(F)}cH_{\bm u_j,a_j}^-\Big)\bigcap\Big(\bigcap_{k\in K_{\mathcal{A}}(F)}cH_{\bm u_k,a_k}^+\Big)\bigcap K_0^+.
\]
Without causing confusion, suppose the linear hyperplane $K_0$ is labelled by $m+1$. The bijection $\kappa$ naturally gives rise to the following bijection
\begin{equation}\label{Bijection}
\epsilon: \Sign(\mathcal{A}) \to \Sign^+(c\mathcal{A}), \quad \Sign_{\mathcal{A}}(F)\mapsto\Sign_{c\mathcal{A}}\big(\kappa(F)\big).
\end{equation}
Based on \eqref{Face-Sign-Vector} and the bijection $\epsilon$, for each $F$ of $\mathcal{A}$, we have
\begin{equation}\label{Sign-Coning}
\Sign_{\mathcal{A}}(F)_i=\Sign_{c\mathcal{A}}\big(\kappa(F)\big)_i\,\For i=1,2,\ldots,m\quad\And\quad \Sign_{c\mathcal{A}}\big(\kappa(F)\big)_{m+1}=+.
\end{equation}

Additionally, since $c\mathcal{A}$ is a linear arrangement, every face $F$ of $c\mathcal{A}$ has an {\em opposite face} $F^{op}$ given by
\[
F^{op}:=\{-\bm x\mid \bm x\in F\},
\]
which is also a face of $c\mathcal{A}$. Naturally, the sign vector of $F^{op}$ is obtained by reversing the sign vector of $F$, that is,
\begin{equation}\label{Sign-Opposite-Face}
\Sign_{c\mathcal{A}}(F^{op})_i=-\Sign_{c\mathcal{A}}(F)_i.
\end{equation}
Let $\mathcal{F}^-(c\mathcal{A})$ be the set of the opposite faces of the faces in $\mathcal{F}^+(c\mathcal{A})$, $O=\bigcap_{i=1,\ldots,m}cH_i\bigcap K_0$, and $G,G^{op}$ be other two faces of $c\mathcal{A}/K_0$. Clearly
\begin{equation}\label{Coning-Face-Number}
\mathcal{F}(c\mathcal{A})=\mathcal{F}^+(c\mathcal{A})\sqcup \mathcal{F}^-(c\mathcal{A})\sqcup\{O,G,G^{op}\}.
\end{equation}
Let $G^+\in\mathcal{F}^+(c\mathcal{A})$ be the closed region bounded by $G$. Then each $i$-th component $\Sign_{c\mathcal{A}}(G)_i$ of the sign vector $\Sign_{c\mathcal{A}}(G)$ equals the $i$-th component $\Sign_{c\mathcal{A}}(G^+)_i$ of the sign vector $\Sign_{c\mathcal{A}}(G^+)$ except for $\Sign_{c\mathcal{A}}(G)_{m+1}=0$. Together with \eqref{Sign-Coning}, \eqref{Sign-Opposite-Face} and \eqref{Coning-Face-Number}, we conclude that the signs of $\mathcal{A}$ and $c\mathcal{A}$ are uniquely determined by each other. This directly shows that (i) is equivalent to (iii) in \autoref{Three-Type-Sign}, that is, for any vectors $\bm a,\bm b\in\mathbb{R}^m$, the conings $c\mathcal{A}_{\bm a}$ and $c\mathcal{A}_{\bm b}$ are sign equivalent if and only if $\bm a$ and $\bm b$ belong to the same open face of $\delta\mathcal{A}_{\bm o}$.

Further noticing $ c H_{\bm u_i,a_i}=H_{(\bm u_i,a_i)}$, so the elementary lift $\mathcal{A}^{\bm a}$ of $\mathcal{A}_{\bm o}$ is obtained from the coning $c\mathcal{A}_{\bm a}$ of $\mathcal{A}_{\bm a}$ by deleting the linear hyperplane $K_0$, i.e.,
\[
\mathcal{A}^{\bm a}=c\mathcal{A}_{\bm a}\setminus K_0.
\]
Let $O$ be the intersection of all hyperplanes in $c\mathcal{A}_{\bm a}$, $G$ be one of faces of $c\mathcal{A}_{\bm a}/K_0$ that does not lie in any $cH_{\bm u_i,a_i}$, and $G^+\in\mathcal{F}^+(c\mathcal{A}_{\bm a})$ and $G^-$ be closed regions bounded by $G$.
First note that if $\mathcal{A}^{\bm a}$ contains only the single hyperplane $H_{(\bm u_1,a_1)}$, then $\mathcal{F}(\mathcal{A}^{\bm a})$ consists of the liner hyperplane $H_{(\bm u_1,a_1)}$ and the two half-spaces $H_{(\bm u_1,a_1)}^+$, and $H_{(\bm u_1,a_1)}^-$. When $|\mathcal{A}^{\bm a}|\ge 2$,
comparing $\mathcal{A}_{\bm a}$ and $c\mathcal{A}_{\bm a}$,  the only minor change to the face poset $\mathcal{F}(\mathcal{A}_{\bm a})$ is as follows: $\mathcal{F}(\mathcal{A}^{\bm a})$ is obtained from  $\mathcal{F}(c\mathcal{A}_{\bm a})$ by removing its seven members $O,G,G^+,G^-,G^{op},(G^+)^{op}$ and $(G^-)^{op}$, and adding new ones $\bar{G}=G^+\cup G\cup G^-$, $\bar{G}^{op}=(G^+)^{op}\cup G^{op}\cup(G^-)^{op}$ and $O^{\bm a}$, where $O^{\bm a}$ is the intersection of all hyperplanes in $\mathcal{A}^{\bm a}$. Namely, when $|\mathcal{A}^{\bm a}|\ge 2$, we have
\begin{equation}\label{Elementary-Lift-Face-Number}
\mathcal{F}(\mathcal{A}^{\bm a})=\mathcal{F}^{+}(c\mathcal{A}_{\bm a})\setminus\{G^+,(G^-)^{op}\}\sqcup\mathcal{F}^{-}(c\mathcal{A}_{\bm a})\setminus\{G^-,(G^+)^{op}\}\sqcup\{\bar{G},\bar{G}^{op},O^{\bm a}\}.
\end{equation}
Immediately, we deduce
\[
\Sign(\mathcal{A}^{\bm a})=\Big\{\big(\Sign_{c\mathcal{A}}(F)\big)_{[m]}\mid F\in\mathcal{F}(c\mathcal{A}_{\bm a})\Big\},
\]
where $\big(\Sign_{c\mathcal{A}}(F)\big)_{[m]}$ is the subvector of $\Sign_{c\mathcal{A}}(F)$ with indices in $[m]$. Thus the result that (i) are equivalent (iv) in \autoref{Three-Type-Sign} each other, is directly established  by that (i) is equivalent to (iii) in \autoref{Three-Type-Sign}.

Below we present a small example to illustrate relations among the faces of the parallel translation $\mathcal{A}_{\bm a}$, coning $c\mathcal{A}_{\bm a}$ and elementary lift $\mathcal{A}^{\bm a}$.
\begin{example}
{\rm Let $\mathcal{A}_{\bm o}=\{H_{\bm u_1}=0,H_{\bm u_2}=0\}$ be a hyperplane arrangement in $\mathbb{R}$. A parallel translation $\mathcal{A}_{\bm a}$,  a coning $c\mathcal{A}_{\bm a}$, an elementary lift $\mathcal{A}^{\bm a}$ of $\mathcal{A}_{\bm o}$,  and the sign vectors of their regions are illustrated below.
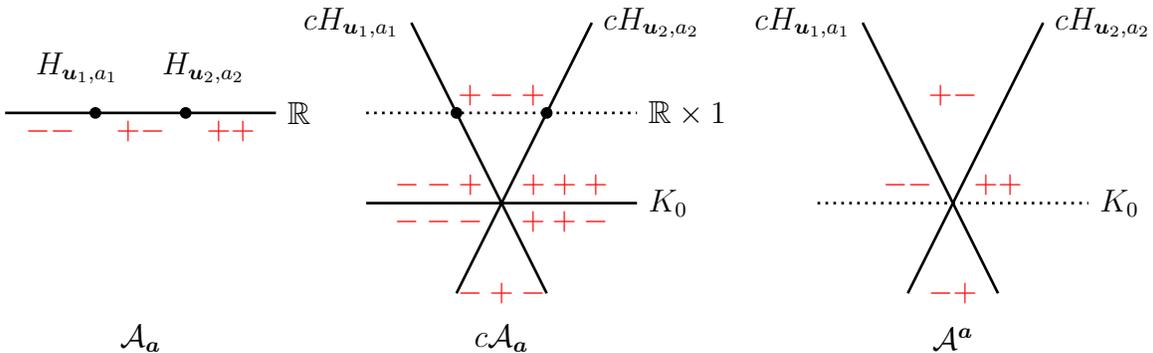
\begin{figure}[H]
\centering
\begin{tikzpicture}[scale=1.2,line width=1pt]
\draw [black,-](0,2) -- (3,2) node [right, black] {$\mathbb{R}$};
\draw (1.5,-0.5) node {$\mathcal{A}_{\bm a}$};
\draw (0.8,2.5) node {$H_{\bm u_1,a_1}$};
\draw (2.2,2.5) node {$H_{\bm u_2,a_2}$};
\draw (0.5,1.8) node[red] {$--$};
\draw (1.5,1.8) node[red] {$+-$};
\draw (2.5,1.8) node[red] {$++$};
\filldraw (1,2)circle (.05) (2,2)circle (.05);
\draw [black,-](4,1) -- (7,1) node [right, black] {$K_0$};
\draw [black,dotted](4,2) -- (7,2) node [right, black] {$\mathbb{R}\times 1$};
\filldraw (5,2)circle (.05) (6,2)circle (.05);
\draw[black,-](5,0)--(6.5,3) node [right, black] {$cH_{\bm u_2,a_2}$};
\draw[black,-](6,0)--(4.5,3) node [left, black] {$cH_{\bm u_1,a_1}$};
\draw (4.8,1.2) node[red] {$--+$};
\draw (5.5,2.2) node[red] {$+-+$};
\draw (6.2,1.2) node[red] {$+++$};
\draw (4.8,0.8) node[red] {$---$};
\draw (5.5,0) node[red] {$-+-$};
\draw (6.2,0.8) node[red] {$++-$};
\draw (5.5,-0.5) node {$c\mathcal{A}_{\bm a}$};
\draw [black,dotted](9,1) -- (12,1) node [right, black] {$K_0$};
\draw[black,-](10,0)--(11.5,3) node [right, black] {$cH_{\bm u_2,a_2}$};
\draw[black,-](11,0)--(9.5,3) node [left, black] {$cH_{\bm u_1,a_1}$};
\draw (10,1.2) node[red] {$--$};
\draw (10.5,2.2) node[red] {$+-$};
\draw (11,1.2) node[red] {$++$};
\draw (10.5,0) node[red] {$-+$};
\draw (10.5,-0.5) node {$\mathcal{A}^{\bm a}$};
\end{tikzpicture}
\end{figure}
}
\end{example}
Let us return to general case. When $\mathcal{A}_{\bm a}$ contains at least two distinct hyperplanes, by summarizing \eqref{Coning-Face-Number} and \eqref{Elementary-Lift-Face-Number},  the face numbers of parallel translations, conings and elementary lifts have the following relationships
\[
|\mathcal{F}(c\mathcal{A}_{\bm a})|=2|\mathcal{F}(\mathcal{A}_{\bm a})|+3\quad\And\quad |\mathcal{F}(\mathcal{A}^{\bm a})|=2|\mathcal{F}(\mathcal{A}_{\bm a})|-1.
\]

Furthermore, we denote by $\Sigma_p$, $\Sigma_c$ and $\Sigma_e$ the corresponding configuration spaces of parallel translation, coning, elementary lift of $\mathcal{A}_{\bm o}$ in turn, that is,
\[
\Sigma_p:=\{\mathcal{A}_{\bm a}:\bm a\in\mathbb{R}^m\},\quad\Sigma_c:=\{c\mathcal{A}_{\bm a}:\bm a\in\mathbb{R}^m\}\quad\And\quad \Sigma_e:=\{\mathcal{A}^{\bm a}:\bm a\in\mathbb{R}^m\}.
\]
We use notations  $\stackrel{c}\sim$, $\stackrel{n}\sim$ and $\stackrel{s}\sim$ to denote the combinatorial, normal and sign equivalence relations on real hyperplane arrangements, respectively. Integrating \autoref{Three-Type-Normal-Combinatorial}, \autoref{Three-Type-Sign} and \autoref{Geometric-Face}, we conclude that any one of the numbers of the three types of equivalence classes in $\Sigma_p$, $\Sigma_c$ and $\Sigma_e$ do not exceed the number of faces of $\delta\mathcal{A}_{\bm o}$.
\begin{corollary} With the above notations, we have
\begin{itemize}
\item [{\rm(a)}] $|\Sigma_p/\stackrel{c}\sim|\le |\Sigma_p/\stackrel{n}\sim|\le |\Sigma_p/\stackrel{s}\sim|=|\mathcal{F}(\mathcal{A}_{\bm o})|$.
\item [{\rm(b)}] $|\Sigma_c/\stackrel{c}\sim|\le |\Sigma_c/\stackrel{s}\sim|=|\mathcal{F}(\mathcal{A}_{\bm o})|$.
\item [{\rm(c)}] $|\Sigma_e/\stackrel{c}\sim|\le |\Sigma_e/\stackrel{s}\sim|=|\mathcal{F}(\mathcal{A}_{\bm o})|$.
\end{itemize}
In particular, when $\mathcal{A}_{\bm o}$ is not a multi-arrangement, we have
\begin{itemize}
\item [{\rm(d)}] $|\Sigma_p/\stackrel{n}\sim|=|\mathcal{F}(\mathcal{A}_{\bm o})|$.
\end{itemize}
\end{corollary}

\section{Oriented matroids}\label{Sec-5}
In this section, we revisit the normal and sign equivalence relations on real hyperplane arrangements from the perspective of (affine) oriented matroids. Oriented matroids, introduced by Bland-Las Vergnas in 1978 \cite{Bland-Lasvergnas1978}, can be thought of as a combinatorial abstraction of real hyperplane arrangements, of polytopes and others. As we are concerned with real hyperplane arrangements, we only consider covectors, and do not discuss other perspectives on oriented matroids. A more general setting and background on oriented matroids can be found in the book \cite{Bjorner1999} by Bj\"{o}rner et al..

Let us start with some necessary terminology. A {\em sign vector} is a function $X:E\to\{+,0,-\}$, that is, an assignment of signs to every element of $E$. The power set $\{+,0,-\}^E$ is the collection of all possible sign vectors and $X_e$ stands for $X(e)$ for all $e\in E$. The {\em support} of $X$ is $\underline{X}:=\{e\in E\mid X_e\ne 0\}$ and its {\em zero set} $z(X):=E\setminus\underline{X}$. The {\em opposite} $-X$ of a sign vector $X$ is defined as $(-X)_e=-(X_e)$. The {\em composition} $X\circ Y$ of two sign vectors $X$ and $Y$ is defined to be
\begin{equation*}
X\circ Y:=
\begin{cases}
X_e,& \mbox{ if } X_e\ne0,\\
Y_e,& \mbox{ otherwise}.
\end{cases}
\end{equation*}
The {\em separation set} of  two sign vectors $X$ and $Y$ is
\[
S(X,Y):=\{e\in E\mid X_e\ne -Y_e\ne0\}.
\]
With these terminologies, we are now ready to define oriented matroids.
\begin{definition}[Covector Axioms]
{\rm
An {\em oriented matroid} $\mathcal{M}=(E,\mathcal{L})$ is an ordered pair of a finite set $E$ and a collection  $\mathcal{L}\subseteq\{\pm,0\}^E$ of sign vectors satisfying:
\begin{itemize}
  \item the zero vector $\bm 0\in\mathcal{L}$,
  \item  $X\in\mathcal{L}$ implies $-X\in\mathcal{L}$,
  \item  $X,Y\in\mathcal{L}$ implies $X\circ Y\in\mathcal{L}$,
  \item  if $X,Y\in\mathcal{L}$ and $e\in S(X,Y)$, then there exists $Z\in\mathcal{L}$ such that $Z_e=0$ and $Z_f=(X\circ Y)_f=(Y\circ X)_f$ for all $f\notin S(X,Y)$.
\end{itemize}
In this sense, the members of $\mathcal{L}$ are called {\em covectors} of $\mathcal{M}$.
}
\end{definition}

Next, we describe how the motivating model for the covector axiomatization of oriented matroids arises from central real arrangements. Let $\mathcal{A}=\{H_{\bm u_e}\mid e\in E\}$ be a central arrangement in $\mathbb{R}^n$. By a quick review of Subsection 2.3, for each point $\bm x\in\mathbb{R}^n$, the central arrangement $\mathcal{A}$ induces a unique sign vector $\Sign_{\mathcal{A}}(\bm x)\in\{\pm,0\}^E$ based on the position information. Clearly $\Sign_{\mathcal{A}}(\bm x)=\Sign_{\mathcal{A}}(\bm y)$ if and only if $\bm x$ and $\bm y$ come from the same open face $\mathcal{A}$. Thus the set consisting of the sign vectors $\Sign_{\mathcal{A}}(\bm x)$ for all points $\bm x\in\mathbb{R}^n$, has only finitely many members, which agrees with $\Sign(\mathcal{A})$ via \eqref{Arrangement-Sign}. Notably, $\Sign(\mathcal{A})$ exactly obeys the covector axiomatization of oriented matroids. Hence, the order pair $\big(E,\Sign(\mathcal{A})\big)$ automatically forms an oriented matroid, denoted by $\mathcal{M}(\mathcal{A})$.

In general, affine oriented matroids can be thought of an abstraction of affine real arrangements. Suppose $\mathcal{\mathcal{A}}=\{H_{\bm u_e,a_e}\mid e\in E\}$ is an affine arrangement in $\mathbb{R}^n$. Recall from \autoref{Sec-4-3} that $\mathcal{\mathcal{A}}$ leads to a linear arrangement
\[
c\mathcal{A}=\{cH_{\bm u_e,a_e}\mid e\in E\}\sqcup\{H_{\bm u_g}=K_0:x_{n+1}=0\}
\]
in $\mathbb{R}^{n+1}$ by coning.  Observe from the bijection $\epsilon$ in \eqref{Bijection}, \eqref{Sign-Coning} and \eqref{Coning-Face-Number} that
\[
\Sign(\mathcal{A})=\big\{\bm s_{E}\mid \bm s\in\Sign(c\mathcal{A})\And \bm s_g=+\big\}.
\]
Based on the important observation, affine oriented matroids can be naturally obtained from oriented matroids.  From \autoref{Affinie-Oriented-Matroid}, the order pair $\big(E,\Sign(\mathcal{A})\big)$ gives an affine oriented matroids on $E$, denoted by $\mathcal{M}(\mathcal{A})$ as well.
\begin{definition}[Affine Oriented Matroid]\label{Affinie-Oriented-Matroid}
{\rm
Let $\big(E\sqcup\{g\},\mathcal{L}\big)$ be an oriented matroid, $g$ be a distinguished element that is not a loop. The order pair $(E,\mathcal{W}\)$ is an {\em affine oriented matroid} on $E$, where
\[
\mathcal{W}:=\big\{X_E:=(X_e)_{e\in E}\mid X\in\mathcal{L}\And X_g=+\big\}.
\]
}
\end{definition}
In an unpublished manuscript 1992, Karlander \cite{Karlander1992} has given an axiomatization of affine oriented matroids, which provides an intrinsic characterization for such affine sign vector system $\mathcal{W}$, not depending on an ambient oriented matroid. Recently, Baum and Zhu \cite{Baum-Zhu2018} further reassessed Karlander's axiomatization of affine oriented matroids.

Two (affine) oriented matroids $\mathcal{M}$ and $\mathcal{M}'$ are {\em equivalent} if they have the same covectors, denoted by $\mathcal{M}\sim\mathcal{M}'$. Immediately, the connection between face posets of real hyperplane arrangements and covectors of (affine) oriented matroids reveals the following fact.
\begin{fact}\label{Arrangement-Oriented-Matroid}
{\rm
Let $\mathcal{A}=\{H_e\mid e\in E\}$ and $\mathcal{A}'=\{H'_e\mid e\in E\}$ be hyperplane arrangements in $\mathbb{R}^n$. Then $\mathcal{A}$ and $\mathcal{A}'$ are sign equivalent if and only if  the corresponding  (affine) oriented matroids $\mathcal{M}(\mathcal{A})$ and $\mathcal{M}(\mathcal{A}')$ are equivalent.
}
\end{fact}

For convenience, suppose $E=[m]$ and $g=m+1$ later. With \autoref{Geometric-Sign} and \autoref{Arrangement-Oriented-Matroid}, we have the following result.
\begin{corollary}\label{Geometric-Equivalent}
Let $\bm a,\bm b\in\mathbb{R}^m$. Then $\mathcal{A}_{\bm a}$ and $\mathcal{A}_{\bm b}$ are normally equivalent if and only if $\mathcal{M}(\mathcal{A}_{\bm a})$ and $\mathcal{M}(\mathcal{A}_{\bm b})$ are equivalent up to relabeling.
\end{corollary}

Furthermore, the following is a direct consequence of \autoref{Geometric-Face} and \autoref{Arrangement-Oriented-Matroid}.
\begin{corollary}
Suppose $\mathcal{A}_{\bm o}$ is not a multi-arrangement. Let $\bm a,\bm b\in\mathbb{R}^m$. Then $\mathcal{A}_{\bm a}$ and $\mathcal{A}_{\bm b}$ are normally equivalent if and only if $\mathcal{M}(\mathcal{A}_{\bm a})$ and $\mathcal{M}(\mathcal{A}_{\bm b})$ are equivalent.
\end{corollary}

Applying \autoref{Arrangement-Oriented-Matroid} to \autoref{Three-Type-Sign}, we immediately obtain the next corollary.
\begin{corollary}\label{Same-Three-Orientedmatroids}
Let $\bm a,\bm b\in\mathbb{R}^m$. Then the following are equivalent:
\begin{itemize}
  \item [\i] $\bm a$ and $\bm b$ belong to the same open face of $\delta\mathcal{A}_{\bm o}$;
  \item [\ii] $\mathcal{M}(\mathcal{A}_{\bm a})$ and $\mathcal{M}(\mathcal{A}_{\bm b})$ are equivalent;
  \item [\iii]$\mathcal{M}(c\mathcal{A}_{\bm a})$ and $\mathcal{M}(c\mathcal{A}_{\bm b})$ are equivalent;
  \item [\iv] $\mathcal{M}(\mathcal{A}^{\bm a})$ and $\mathcal{M}(\mathcal{A}^{\bm b})$ are equivalent.
\end{itemize}
\end{corollary}

Let
\[
\Omega_p:=\big\{\mathcal{M}(\mathcal{A}_{\bm a}):\bm a\in\mathbb{R}^m\big\},\quad \Omega_c:=\big\{\mathcal{M}(c\mathcal{A}_{\bm a}):\bm a\in\mathbb{R}^m\big\},\quad \Omega_e:=\big\{\mathcal{M}(\mathcal{A}^{\bm a}):\bm a\in\mathbb{R}^m\big\}.
\]
The next result that the numbers of different (affine) oriented matroids in $\Omega_p$, $\Omega_c$ and $\Omega_e$ agree with the number of faces of $\delta\mathcal{A}_{\bm o}$, is a direct consequence of  \autoref{Same-Three-Orientedmatroids}.
\begin{corollary}
With the above notations, we have
  \[
  |\Omega_p/\sim|= |\Omega_c/\sim|= |\Omega_e/\sim|=|\mathcal{F}(\delta\mathcal{A}_{\bm o})|.
  \]
\end{corollary}
\section{Derived arrangement}\label{Sec-6}
The derived arrangement is closely related to combinatorics, algebra, topology, and other fields. For more detailed information, please refer to Athanasiadis \cite{Athanasiadis1999}, Bayer-Brandt \cite{Bayer-Brandt1997}, Cohen-Falk-Randell \cite{Cohen-Falk-Randell2020}, Falk \cite{Falk1994}, Libgober-Settepanella \cite{Libgober-Settepanella2018}, Sawada-Settepanella-Yamagata \cite{Sawada2017}, Yamagata \cite{Yamagata2023}, Ziegler \cite{Ziegler1993}, and others. In this section, we provide several new characterizations for real derived arrangements associated with the faces and sign vectors.

Let us first define two operators related to real hyperplane arrangements: face operator and sign operator. Associated with the linear arrangement $\mathcal{A}_{\bm o}$, we define a {\em sign operator} on each subset $S\subseteq\mathbb{R}^m$ as
\begin{equation}\label{Sign-Operator}
\sign_{\mathcal{A}_{\bm o}}(S):=\bigcup_{\bm a\in S}\Sign(\mathcal{A}_{\bm a}).
\end{equation}
Conversely, for any vector $\bm s\in\{\pm,0\}^m$, it yields a subset
\begin{equation}\label{Face-Operator}
\face_{\mathcal{A}_{\bm o}}(\bm s):=\big\{\bm a\in\mathbb{R}^m\mid \bm s\in\Sign(\mathcal{A}_{\bm a})\big\}
\end{equation}
in  $\mathbb{R}^n$. We further define a {\em face operator} on each subset $S\subseteq\{\pm,0\}^m$ to be
\begin{equation*}\label{Region-Operator}
\face_{\mathcal{A}_{\bm o}}(S):=\bigcap_{\bm s\in S}\face_{\mathcal{A}_{\bm o}}(\bm s).
\end{equation*}
First note the following fact from \autoref{Separate-Cells}.
\begin{fact}\label{Fact}
{\rm
Let $F_1,F_2$ be distinct faces of $\delta\mathcal{A}_{\bm o}$, and $\bm a\in\relint(F_1),\bm b\in\relint(F_2)$. Then
\[
\Sign(\mathcal{A}_{\bm a})\nsubseteq\Sign(\mathcal{A}_{\bm b})\quad\And\quad\Sign(\mathcal{A}_{\bm b})\nsubseteq\Sign(\mathcal{A}_{\bm a}).
\]
}
\end{fact}

To obtain our results, the next lemma is necessary.
\begin{lemma}\label{Face-Characterization-1}
Let $F\in\mathcal{F}(\delta\mathcal{A}_{\bm o})$. Then
\[
\relint(F)=\face_{\mathcal{A}_{\bm o}}\big(\sign_{\mathcal{A}_{\bm o}}(\relint(F))\big).
\]
\end{lemma}
\begin{proof}
For any $\bm s\in\sign_{\mathcal{A}_{\bm o}}\big(\relint(F)\big)$, clearly $\relint(F)\subseteq \face_{\mathcal{A}_{\bm o}}(\bm s)$, which means
\[
\relint(F)\subseteq\bigcap_{\bm s\in\sign_{\mathcal{A}_{\bm o}}\big(\relint(F)\big)}\face_{\mathcal{A}_{\bm o}}(\bm s)=\face_{\mathcal{A}_{\bm o}}\big(\sign_{\mathcal{A}_{\bm o}}(\relint(F))\big).
\]
We now prove the opposite side. According to (i) and (ii) in \autoref{Three-Type-Sign}, we deduce
\[
\sign_{\mathcal{A}_{\bm o}}\big(\relint(F)\big)=\Sign(\mathcal{A}_{\bm x}),\quad\forall\, \bm x\in\relint(F).
\]
Following this,  for a fixed $\bm a\in\relint(F)$, we have
\[
\face_{\mathcal{A}_{\bm o}}\big(\sign_{\mathcal{A}_{\bm o}}(\relint(F))\big)=\face_{\mathcal{A}_{\bm o}}\big(\Sign(\mathcal{A}_{\bm a})\big).
\]
This further implies $\Sign(\mathcal{A}_{\bm a})\subseteq\Sign(\mathcal{A}_{\bm x})$ for any $\bm x\in\face_{\mathcal{A}_{\bm o}}\big(\sign_{\mathcal{A}_{\bm o}}(\relint(F))\big)$. Together with \autoref{Fact}, we conclude
\[
\Sign(\mathcal{A}_{\bm a})=\Sign(\mathcal{A}_{\bm x}),\quad\forall\,\bm x\in\face_{\mathcal{A}_{\bm o}}\big(\sign_{\mathcal{A}_{\bm o}}(\relint(F))\big).
\]
Once again, we derive from (i) and (ii) in \autoref{Three-Type-Sign}  that each $\bm x\in\face_{\mathcal{A}_{\bm o}}\big(\sign_{\mathcal{A}_{\bm o}}(\relint(F))\big)$ belongs to $\relint(F)$. Consequently, we have $\relint(F)=\face_{\mathcal{A}_{\bm o}}\big(\sign_{\mathcal{A}_{\bm o}}(\relint(F))\big)$.
\end{proof}
With \autoref{Face-Characterization-1}, we give an alternative characterization of the open faces of real derived arrangements.
\begin{theorem}\label{Two-Operators}
Let $S$ be a nonempty subset in $\mathbb{R}^m$. Then  $\face_{\mathcal{A}_{\bm o}}\big(\sign_{\mathcal{A}_{\bm o}}(S)\big)=S$ if and only if $S$ is a open face of $\delta\mathcal{A}_{\bm o}$.
\end{theorem}
\begin{proof}
\autoref{Face-Characterization-1} indicates that the sufficiency holds. For necessity, given a point $\bm a\in S$, we assert $S=\relint(F_{\bm a})$ with $F_{\bm a}\in\mathcal{F}(\delta\mathcal{A}_{\bm o})$. From (i) and (ii) in \autoref{Three-Type-Sign}, we have
\[
\sign_{\mathcal{A}_{\bm o}}\big(\relint(F_{\bm a})\big)=\Sign(\mathcal{A}_{\bm a}).
\]
Applying \autoref{Face-Characterization-1} again, we derive $S=\relint(F_{\bm a})$, which  completes the proof.
\end{proof}
Furthermore, \autoref{Two-Operators} directly leads to the following result.
\begin{corollary}\label{New-Derived-Arrangement}
Let $\mathcal{A}$ be a hyperplane arrangement in $\mathbb{R}^m$. If $\face_{\mathcal{A}_{\bm o}}\big(\sign_{\mathcal{A}_{\bm o}}(\relint(F))\big)=\relint(F)$ for any face $F\in\mathcal{F}(\mathcal{A})$, then $\mathcal{A}$ coincides with the derived arrangement $\delta\mathcal{A}_{\bm o}$ of $\mathcal{A}_{\bm o}$.
\end{corollary}

Notably, by uniformly substituting $\mathcal{A}_{\bm a}$ with either $c\mathcal{A}_{\bm a}$ or $\mathcal{A}^{\bm a}$ in \eqref{Sign-Operator} and \eqref{Face-Operator}, \autoref{Three-Type-Sign} ensures that \autoref{Two-Operators} and \autoref{New-Derived-Arrangement} hold in both cases.
\section*{Acknowledgements}
The first author is supported by National Natural Science Foundation of China under Grant No. 12301424. The third author is supported by National Natural Science Foundation of China under Grant No. 12171114.

\end{document}